\documentclass[11pt,reqno]{amsart}
\usepackage{enumerate}
\usepackage{amsrefs}
\usepackage{amsfonts, amsmath, amssymb, amscd, amsthm, bm, cancel,mathrsfs}
\usepackage{url}
\usepackage{graphicx}
\usepackage[
linktocpage=true,colorlinks,citecolor=magenta,linkcolor=blue,urlcolor=magenta]{hyperref}
\usepackage{multicol}
\usepackage{comment}
\usepackage[margin=1in]{geometry}
\newtheorem{thm}{Theorem}[section]

\newtheorem*{thmA}{Theorem A}

\newtheorem{conj}{Conjecture}[section]
\newtheorem{lem}[thm]{Lemma}

\theoremstyle{definition}

\newtheorem*{ack}{Acknowledgments}
\theoremstyle{remark}
\newtheorem{rem}{Remark}

\numberwithin{equation}{section}
\allowdisplaybreaks

\renewcommand{\(}{\left(}
\renewcommand{\)}{\right)}

\renewcommand{\a}{\alpha}
\renewcommand{\b}{\beta}
\newcommand{\g}{\gamma}
\renewcommand{\d}{\delta}
\newcommand{\e}{\epsilon}
\renewcommand{\k}{\kappa}
\renewcommand{\l}{\lambda}

\renewcommand{\t}{\theta}
\newcommand{\s}{\sigma}

\newcommand{\G}{\Gamma}

\newcommand{\ra}{\rightarrow}

\newcommand{\Vol}{\operatorname{Vol}}

\begin{document}
\title[Blaschke-Santal\'o type inequalities and quermassintegral inequalities]{Blaschke-Santal\'o type inequalities and quermassintegral inequalities in space forms}
	
\author{Yingxiang Hu}
	\address{School of Mathematics, Beihang University, Beijing 100191, P.R. China}
	\email{\href{mailto:huyingxiang@buaa.edu.cn}{huyingxiang@buaa.edu.cn}}
\author{Haizhong Li}
\address{Department of Mathematical Sciences, Tsinghua University, Beijing 100084, P.R. China}
\email{\href{mailto:lihz@tsinghua.edu.cn}{lihz@tsinghua.edu.cn}}
	

	\begin{abstract}
		In this paper, we prove a family of identities for closed and strictly convex hypersurfaces in the sphere and hyperbolic/de Sitter space. As applications, we prove Blaschke-Santal\'o type inequalities in the sphere and hyperbolic/de Sitter space, which generalizes the previous work of Gao, Hug and Schneider \cite{GHS03}. We also prove the quermassintegral inequalities in hyperbolic/de Sitter space.
	\end{abstract}	
	\maketitle

\section{Introduction}
A convex body $K$ in Euclidean space $\mathbb R^{n+1}$ is a compact convex set with non-empty interior. For any interior point $z\in \operatorname{int}K$, the dual body of $K$ with respect to $z$ is defined by
\begin{align*}
K_z^\ast=\{ y+z | y\in \mathbb R^{n+1} : y\cdot (x-z)\leq 1, \forall ~x\in K \}.
\end{align*}
As noted by Santal\'o \cite{Sant1949}, there exists a unique point $s=s(K)\in \operatorname{int}K$ such that $\Vol(K_s^\ast)\leq \Vol(K_z^\ast)$ for all $z\in \operatorname{int}K$. The famous Blaschke-Santal\'o inequality says that
\begin{align}\label{s1:Blaschke-Santalo-ineq}
\Vol(K) \cdot \Vol(K_s^\ast) \leq |\mathbb B^{n+1}|^2,
\end{align}
where $|\mathbb B^{n+1}|$ is the volume of the unit ball in $\mathbb R^{n+1}$. Equality holds if and only if $K$ is an ellipsoid. This inequality in $\mathbb R^3$ was first proved by Blaschke \cite{Blaschke1917}, and later it is generalized to higher dimensions by Santal\'o \cite{Sant1949} with the equality case characterized by Petty \cite{Petty1985}. For sake of its importance, more proofs of the Blaschke-Santal\'o inequality were given by Hug \cite{Hug1996}, Meyer and Pajor \cite{Meyer-Pojor1990}. Andrews \cite{And96} studied the affine normal flow, which also provides a proof of the Blaschke-Santal\'o inequality for smooth convex hypersurfaces. By using Steiner symmetrization, Lutwak and Zhang \cite{Lutwak-Zhang1997} extended the Blaschke-Santal\'o inequality to the $L_p$ Blaschke-Santal\'o inequality for star bodies.

The volume functional can be considered as one of the quermassintegrals, which are the coefficients in the {\em Steiner formula}:
\begin{align*}
\Vol(K+tB)=\sum_{i=0}^{n+1}\binom{n+1}{i}W_i(K)t^i,
\end{align*}
where $K+tB=\{ x\in \mathbb R^{n+1} : d(K,x)\leq t \}$ is the outer parallel body of $K$ at distance $t\geq 0$. If $K$ is smooth, then the quermassintegrals can be also expressed in terms of curvature integrals of its boundary and the enclosed volume as follows:
\begin{equation}\label{s1:quermassintegral-Euclidean-space}
\begin{split}
&W_0(K)=\operatorname{Vol}(K), \quad W_{k+1}(K)=\frac{1}{n+1}\int_{\partial K}E_{k}d\mu, \quad k=0,1,\cdots,n,
\end{split}
\end{equation}
where $E_k$ is the normalized $k$th mean curvature and $d\mu$ is the area element of $\partial K$, respectively. For simplicity, we define the $k$th {\em mean radius} of the smooth bounded domain $K$ as follows:
\begin{align}\label{s1:def-mean-radius}
\zeta_k(K)=f_k^{-1}(W_k(K)),
\end{align}
where $f_k:[0,\infty)\ra \mathbb R_{+}$ is a monotone function defined by
$$
f_k(r)=W_k(B_r)=\frac{1}{n+1}\int_{\partial B_r}E_{k-1} d\mu=\frac{\omega_n}{n+1}r^{n+1-k},
$$
the $k$th quermassintegral for the ball of radius $r$, $\omega_n$ is the area of the unit sphere $\mathbb S^n\subset \mathbb R^{n+1}$ and $f_k^{-1}$ is the inverse of $f_k$. In terms of the mean radius, the Blaschke-Santal\'o inequality \eqref{s1:Blaschke-Santalo-ineq} becomes
\begin{align}\label{s1:BS-ineq-equivalent-form}
\zeta_0(K) \cdot \zeta_0(K_s^\ast) \leq 1.
\end{align}

The following inequalities can be considered as a natural extension of the Blaschke-Santal\'o inequality \eqref{s1:BS-ineq-equivalent-form} in Euclidean space.
\begin{thmA}
	Let $K$ be a convex body in $\mathbb R^{n+1}$ and $K_z^\ast$ is its dual body of $K$ with respect to any point $z\in \operatorname{int}K$. Then for $0\leq i,j\leq n$ satisfying $i+j\geq n$, there holds
	\begin{align}\label{s1:Firey-ineq}
	W_i(K)^{n+1-j} W_j(K_z^\ast)^{n+1-i} \geq |\mathbb B^{n+1}|^{2n+2-i-j},
	\end{align}
	or equivalently, 
	\begin{align}\label{s1:Firey-ineq-equivalent-form}
	\zeta_{i}(K) \cdot \zeta_{j}(K_z^\ast) \geq 1.
	\end{align}	
	Equality holds if and only if $K$ is a ball centered at $z$.
\end{thmA}
For $i+j=n$, this result is due to Firey \cite{Firey1973} and for $i=j=n$ to Lutwak \cite{Lutwak1975}. The general result was deduced from Firey's case, independently by Heil \cite{Heil1976} and Lutwak \cite{Lutwak1976}. Further inequalities of this type were proved by Ghandehari \cite{Ghandehari1991}, see also Schneider's book \cite[P. 393]{Schneider2014}.

In the sphere or hyperbolic space, the Steiner formula also makes sense, but leads to different functionals. An investigation of such functionals within integral geometry can be found in Santal\'o's book \cite{Sant2004}. In the smooth category, convex bodies and their dual bodies in the sphere, as well as the counterparts in hyperbolic/de Sitter space, are well-defined in terms of the Gauss map (see \S \ref{sec:2.2} for details). The motivation of this paper is to extend Blaschke-Santal\'o's inequality \eqref{s1:BS-ineq-equivalent-form} and the inequalities \eqref{s1:Firey-ineq-equivalent-form} in the sphere and hyperbolic/de Sitter space. The quermassintegrals $W_k$ in Riemannian space forms and the upper branch of de Sitter space are introduced in \eqref{s1:quermassintegral-def}, \eqref{s1:quermassintegral-def-de-Sitter} respectively, and the mean radius $\zeta_k$ can be similarly defined by \eqref{s3:def-mean-radius-spaceforms}.

The first result of this paper is a complete family of identities for convex bodies and their dual bodies in the sphere, where the dual body of a convex body in the sphere is defined by \eqref{s2:dual-body-sphere}.
\begin{thm}\label{thm-sphere}
	Let $K$ be a smooth bounded and strictly convex domain in $\mathbb S^{n+1}$ and $K^\ast$ the dual body of $K$. Then the following identities hold:
	\begin{align}\label{s1:identity-sphere}
	\tan\(\zeta_{n-k}(K)\)\cdot \tan\(\zeta_{k}(K^\ast)\)=1, \quad k=0,1,\cdots,n.
	\end{align}	
\end{thm}

We also prove a complete family of identities for convex bodies in hyperbolic space and their dual bodies in de Sitter space, where the dual body in de Sitter space is defined by \eqref{s2:dual-body-deSitter}.
\begin{thm}\label{thm-hyperbolic}
	Let $K$ be a smooth bounded and strictly convex domain in $\mathbb H^{n+1}$ and $K^\ast$ the dual body of $K$ in $\mathbb S^{n,1}_{+}$. Then the following identities hold:
	\begin{align}\label{s1:identity-hyperbolic}
	\coth(\zeta_{k}(K)) \cdot \tanh(\zeta_{n-k}(K^\ast))=1, \quad k=0,1,\cdots,n.
	\end{align}	
\end{thm}

As a direct application of Theorem \ref{thm-sphere}, we obtain the following Blaschke-Santal\'o type inequalities in the sphere.
\begin{thm}\label{cor-BS-ineq-sphere-extension}
	Let $K$ be a smooth bounded and strictly convex domain in $\mathbb S^{n+1}$ and $K^\ast$ the dual body of $K$. 
	\begin{enumerate}[(i)]
		\item For $0<k\leq n$, there hold
	\begin{align}
	\tan\(\zeta_n(K)\)\cdot \tan\(\zeta_k(K^\ast)\) \geq& 1,    \label{s1:cor-sphere-ineq-1}\\
	\tan\(\zeta_{n-k}(K)\) \cdot \tan\(\zeta_0(K^\ast)\)\leq & 1. \label{s1:cor-sphere-ineq-2}  
	\end{align}
        \item For $0\leq l\leq n$, if $0\leq 2i\leq l$ and $0\leq 2j\leq n-l$, then there hold
	\begin{align}
	\tan\(\zeta_{n-l+2i}(K)\)\cdot \tan\(\zeta_l(K^\ast)\) \geq& 1, \label{s1:cor-sphere-ineq-3} \\ \tan\(\zeta_{n-l-2j}(K)\)\cdot\tan\(\zeta_l(K^\ast)\) \leq& 1. \label{s1:cor-sphere-ineq-4}
	\end{align}
\end{enumerate}
	Equality holds if and only if $K$ is a geodesic ball. 
\end{thm}
\begin{rem}
	It should be mentioned that Gao, Hug and Schneider \cite[P. 166]{GHS03} proved the following Blaschke-Santal\'o type inequality in the sphere: Let $K$ be a convex body in $\mathbb S^{n+1}$ and $K^\ast$ the dual body of $K$. Let $B_r$ be a geodesic ball with $\Vol(K)=\Vol(B_r)$ and $B_r^\ast=B_{\frac{\pi}{2}-r}$ the dual body of $B_r$. Then $\Vol(K^\ast)\leq \Vol(B_r^\ast)$, with equality holds if and only if $K$ is a geodesic ball. This inequality is equivalent to
    \begin{align*}
    \tan\(\zeta_0(K)\) \cdot \tan(\zeta_0(K^\ast)) \leq 1,
    \end{align*}
    which is a special case ($k=n$) of the inequality \eqref{s1:cor-sphere-ineq-2}.
\end{rem}

Similarly, we apply Theorem \ref{thm-hyperbolic} to obtain the complete family of Blaschke-Santal\'o type inequalities in hyperbolic/de Sitter space.
\begin{thm}\label{cor-BS-ineq-hyperbolic-extension}
	Let $K$ be a smooth bounded h-convex\footnote{A smooth bounded domain in hyperbolic space is called {\em h-convex} (which is short for {\em horospherically convex}) if the principal curvatures of its boundary satisfy $\k_i\geq 1$ for all $i$.} domain in $\mathbb H^{n+1}$ and $K^{\ast}$ the dual body of $K$ in $\mathbb S^{n,1}_{+}$. Then for any $0\leq k,l\leq n$, there hold
\begin{align}\label{s1:cor-hyperbolic-ineq-1}
\coth(\zeta_{k}(K)) \cdot \tanh(\zeta_l(K^\ast)) \geq 1, \quad \text{if $k+l<n$},
\end{align}
and
\begin{align}\label{s1:cor-hyperbolic-ineq-2}
\coth(\zeta_{k}(K)) \cdot \tanh(\zeta_l(K^\ast)) \leq 1, \quad \text{if $k+l>n$}.
\end{align}
Equality holds if and only if $K$ is a geodesic ball. 
\end{thm}

Another application of Theorem \ref{thm-hyperbolic} is to prove geometric inequalities in hyperbolic/de Sitter space. We first obtain the following quermassintegral inequalities in hyperbolic space.
\begin{thm}\label{cor-BS-ineq-hyperbolic-quermassintegral-ineq}
Let $K$ is a smooth bounded and strictly convex domain in $\mathbb H^{n+1}$. Then there hold
\begin{align}\label{quermassintegral-ineq-hyperbolic-strictly-convex-II}
W_{n-1}(K) \geq f_{n-1} \circ f^{-1}_{l}(W_l(K)), \quad  0\leq l<n-1.
\end{align}
Equality holds if and only if $K$ is a geodesic ball.
\end{thm}
\begin{rem}
	The full quermassintegral inequalities in hyperbolic space
	\begin{align}\label{s1:full-quermassintegral-ineq-hyperbolic-space}
	W_k(K) \geq f_k\circ f_l^{-1}(W_{l}(K)), \quad 0\leq l<k\leq n
	\end{align}
	were first proved by Wang and Xia \cite{Wang-Xia14} for smooth bounded h-convex domains, see also a new proof by the authors of this paper with Wei \cite{Hu-Li-Wei2020}. It is challenging to extend these inequalities to merely $(k-1)$-convex\footnote{A smooth hypersurface is called {\em $m$-convex} if its principal curvatures satisfy $\k\in \G_{m}^{+}=\{\k \in \mathbb R^n ~|~ E_i(\k)>0, 1\leq i\leq m \}$. $n$-convex is strictly convex and $1$-convex is usually refered as {\em mean convex}.} and starshaped hypersurfaces in hyperbolic space. In this direction, the second author with Wei and Xiong \cite{LWX14} proved the inequality \eqref{s1:full-quermassintegral-ineq-hyperbolic-space} with $k=3$ and $l=1$ for $2$-convex and star-shaped hypersurfaces. Later, Andrews, Chen and Wei \cite{ACW2018} proved the inequality \eqref{s1:full-quermassintegral-ineq-hyperbolic-space} with $0<k\leq n$ and $l=0$ for hypersurfaces with positive sectional curvature\footnote{A smooth hypersurface in hyperbolic space has {\em positive sectional curvature} if its principal curvatures satisfy $\k_i\k_j>1$ for all distant $i,j$.}. Recently, Brendle, Guan and Li \cite{BGL19} (see also \cite{GL21}) proved the inequality \eqref{s1:full-quermassintegral-ineq-hyperbolic-space} with $k=2$ and $l=1$ for mean convex and star-shaped hypersurface, and the inequality \eqref{s1:full-quermassintegral-ineq-hyperbolic-space} with $k=n$ and $0\leq l<n$ for strictly convex hypersurfaces. The authors of this paper with Andrews proved the inequality \eqref{s1:full-quermassintegral-ineq-hyperbolic-space} with $k=n-1$ and $l=n-1-2i (0<2i<n)$ for strictly convex hypersurfaces. 
\end{rem}

Several special cases of the Blaschke-Santal\'o type inequalities in Theorem \ref{cor-BS-ineq-hyperbolic-extension} can be generalized to the strictly convex domains in hyperbolic/de Sitter space.
\begin{thm}\label{cor-BS-ineq-hyperbolic-extension-II}
	Let $K$ be a smooth bounded and strictly convex domain in $\mathbb S^{n+1}$ and $K^\ast$ the dual body of $K$. Then there hold
	\begin{align}
	\coth(\zeta_{l}(K)) \cdot \tanh(\zeta_1(K^\ast)) \geq &1,  \quad 0\leq l<n-1,  \label{s1:hyperbolic-ineq-1}
	\end{align}
	and
	\begin{align}
	\coth(\zeta_{k}(K)) \cdot \tanh(\zeta_0(K^\ast)) \geq &1,  \quad 0\leq k<n.  \label{s1:hyperbolic-ineq-2}
	\end{align}
	Equality holds if and only if $K$ is a geodesic ball.
\end{thm}
\begin{rem}
	The inequality \eqref{s1:hyperbolic-ineq-2} with $k=0$ is equivalent to the following Blachke-Santal\'o type inequality which was previously proved by Gao, Hug and Schneider \cite{GHS03} via two-point symmetrization: Let $K$ be a convex body in hyperbolic space and $K^\ast$ the dual body in de Sitter space. Let $B_r$ be a geodesic ball with $\Vol(K)=\Vol(B_r)$ and $B_r^\ast$ the dual body of $B_r$. Then $\Vol(K^\ast)\leq \Vol(B_r^\ast)$, with equality holds if and only if $K$ is a geodesic ball.
\end{rem}

Motivated by the quermassintegral inequalities in hyperbolic space, it is natural to investigate the following quermassintegral inequalities in de Sitter space.
\begin{conj}\label{s1:conjecture}
Let $K$ be a smooth bounded domain with $(k-1)$-convex and spacelike boundary $\partial K$ in $\mathbb S_{+}^{n,1}$. Then there hold	
\begin{align}\label{s2:quermassintegral-ineq-conjecture}
W_k(K)\leq f_k \circ f_l^{-1}(W_l(K)), \quad 0\leq l<k\leq n. 
\end{align}
Equality holds if and only if $\partial K$ is isometric to a coordinate slice.	
\end{conj}
\begin{rem}
Two special cases of Conjecture \ref{s1:conjecture} have been completely solved: 
\begin{enumerate}[(i)]
	\item The inequality \eqref{s2:quermassintegral-ineq-conjecture} with $k=1$, $l=0$ was verified by Lambert and Scheuer \cite{Lambert-Scheuer2021}, who proved the isoperimetric inequalities for spacelike domains in generalized Robertson-Walker spaces including de Sitter space.
	\item The inequality \eqref{s2:quermassintegral-ineq-conjecture} with $k=2$, $l=1$ was recently proved by Scheuer \cite{Scheuer19}.
\end{enumerate}
\end{rem}
As an application of Theorem \ref{thm-hyperbolic}, we prove the full quermassintegral inequalities \eqref{s2:quermassintegral-ineq-conjecture} under the stronger assumption.
\begin{thm}\label{cor-BS-ineq-deSitter-quermassintegral-ineq}
	Let $K$ be a smooth bounded domain with spacelike boundary $\partial K$ in $\mathbb S_{+}^{n,1}$. Assume that the principal curvatures of $\partial K$ satisfy $0<\k_i\leq 1$. Then there hold	\begin{align}\label{s2:quermassintegral-ineq-I}
	W_k(K)\leq f_k \circ f_l^{-1}(W_l(K)), \quad 0\leq l<k\leq n. 
	\end{align}
	Equality holds if and only if $\partial K$ is isometric to a coordinate slice.
\end{thm}
The inequality \eqref{s2:quermassintegral-ineq-I} with $1<k\leq n$ and $l=1$ holds for strictly convex hypersurfaces in $\mathbb S_{+}^{n,1}$.
\begin{thm}\label{cor-BS-ineq-deSitter-quermassintegral-ineq-II}
	Let $K$ be a smooth bounded domain with strictly convex and spacelike boundary in $\mathbb S^{n,1}_{+}$. Then there hold
	\begin{align}\label{s2:quermassintegral-de-Sitter-ineq-I}
	W_k(K) \leq f_k\circ f_1^{-1}(W_1(K)), \quad  1<k\leq n,
	\end{align}
	and 
	\begin{align}
	W_{n-1}(K) \leq  &f_{n-1}\circ f_{n-2}^{-1}(W_{n-2}(K)), \label{s2:quermassintegral-de-Sitter-ineq-II}\\ W_{n-1}(K) \leq  &f_{n-1}\circ f_{n-3}^{-1}(W_{n-3}(K)). \label{s2:quermassintegral-de-Sitter-ineq-III}
	\end{align}
	Equality holds if and only if $\partial K$ is isometric to a coordinate slice.
\end{thm}
\begin{rem}
Beyond the quermassintegral inequalities, it is not difficult to apply Theorem \ref{thm-hyperbolic} to transform all geometric inequalities in hyperbolic space (e.g. \cite{GeWW14,HL19,Hu-Li-2022,Wang-Xia14,Hu-Li-Wei2020}) to their counterparts in de Sitter space under appropriate convexity assumptions. This idea was previously utilized by Andrews and the authors of this paper, see \cite[Theorem 1.4, Remark 6.4]{AHL2020}.
\end{rem}

The paper is organized as follows. In \S \ref{sec:2}, we collect the preliminaries on starshaped/spacelike hypersurface in Riemannian/Lorentzian warped product, duality and dual flows in the sphere and in hyperbolic/de Sitter space. In \S \ref{sec:3}, we give the definition of quermassintegrals and their variational formulas in Riemannian space forms and in de Sitter space. In \S \ref{sec:4}, we give the proofs of Theorem \ref{thm-sphere} and Theorem \ref{thm-hyperbolic}. In \S \ref{sec:5}, we give the proofs of Theorems \ref{cor-BS-ineq-sphere-extension}--\ref{cor-BS-ineq-deSitter-quermassintegral-ineq-II}.

\begin{ack} This work was supported by National Key R and D Program of China 2021YFA1001800, NSFC grant No.12101027, NSFC grant No.11831005 and NSFC grant No.12126405.
\end{ack}

\section{Preliminaries}\label{sec:2}
\subsection{Hypersurfaces in semi-Riemannian manifolds}
Let $(N,\-g)$ be an $(n+1)$-dimensional semi-Riemannian manifold. The coordinates $(x^\a)_{0\leq \a\leq n}$ in $N$ are labelled from $0$ to $n$. In general, the coordinate $x^0$ refers to the radial distance to a fixed point or the time function. Let $(\xi^i)$ be the local coordinates system of an open neighbourhood $U$ in $M$, where $M$ is a smooth hypersurface in $N$. Locally, $M$ can be represented by a map $X=X(\xi)=(x^\a(\xi))$, and the induced metric is given by
\begin{align*}
g_{ij}=\-g(X_i,X_j), \quad \text{where}~ X_i=\frac{\partial X}{\partial \xi^i}.
\end{align*}
Since $M$ has codimension $1$, its normal space is spanned by a single vector $\nu$ at each point $x\in M$. We can always define a continuous normal vector field $\nu$ by requiring that 
\begin{align*}
\det\(X_1,\cdots,X_n,\nu\)>0, \quad \forall \xi \in U.
\end{align*}
Moreover, the normal vector $\nu$ is normalized such that
\begin{align*}
\-g(\nu,\nu)=\s=\pm 1.
\end{align*} 

The geometry of a hypersurface $M$ in a semi-Riemannian manifold $N$ is governed by the following basic equations (see \cite[(1.1.6),(1.1.21),(1.1.36),(1.1.37)]{Gerh06}):
\begin{enumerate}[(i)]
	\item The Gauss formula is
\begin{align}\label{s2:Gauss-formula}
X_{ij}=-\s h_{ij} \nu,  
\end{align}
where $X_{ij}$ is the second covariant derivatives of $X$ with respect to the induced metric $g$, and $h_{ij}$ is called the {\em second fundamental form} of $M$ with respect to $-\s \nu$.
\item The Weingarten equation is
\begin{align}\label{s2:Weingarten-eq}
\nu_{i}=h_i^k X_k,
\end{align}
where $h_i^k=g^{kl}h_{il}$ is the Weingarten matrix, with its eigenvalues $\k=(\k_1,\cdots,\k_n)$ being called the {\em principal curvatures} of $M$ in $N$.
\item The Codazzi equation is
\begin{align}\label{s2:Codazzi-formula}
h_{ij;k}-h_{ik;j}=\-R_{\a\b\g\d} \nu^\a X_i^\b X_j^\g X_k^\d,
\end{align}
where $h_{ij;k}=\nabla_k h_{ij}$ and $\-R_{\a\b\g\d}$ is the components of Riemann curvature tensor.
\item The Gauss equation is
\begin{align}\label{s2:Gauss-equation}
R_{ijkl}=\s(h_{ik}h_{jl}-h_{jk}h_{il})+\-R_{\a\b\g\d}X_i^\a X_j^\b X_k^\g X_l^\d.
\end{align}
\end{enumerate}
\begin{rem}
	The choice of the normal $\nu$ in the Gaussian formula \eqref{s2:Gauss-formula} is free, i.e., we could just as well have replaced $\nu$ by $-\nu$, then the principal curvatures $\k_i$ would have been replaced by $-\k_i$. We make the following convention on the choice of $\nu$:
    \begin{enumerate}[(1)]
    	\item If $M$ is a closed oriented hypersurface and the ambient space is Riemannian, then in this case we always choose $\nu$ as the outward normal to $M$. Then the starshapedness of $M$ is equivalent to $\-g(\partial_r,\nu)>0$. 
    	\item If $M$ is spacelike (i.e., the induced metric $g$ of $M$ is Riemannian) and the ambient space is Lorentzian, and the coordinate system $(x^\a)$ is supposed to be {\em future directed}, i.e., the time function $x^0$ is increasing on future directed curves. Then in this case we always choose $\nu$ to be also future directed (timelike) normal, which is equivalent to $\-g(\partial_r,\nu)<0$.		
    \end{enumerate}
\end{rem}
	
In the sequel, the Euclidean space $\mathbb R^{n+2}$ is the $(n+2)$-dimensional vector space equipped with the bilinear form
\begin{align*}
\langle v,w\rangle=\sum_{\a=0}^{n+1}v^\a w^\a,
\end{align*}
while the Minkowski space $\mathbb R^{n+1,1}$ is the $(n+2)$-dimensional vector space equipped with the bilinear form
\begin{align*}
\langle v,w\rangle=-v^0w^0+\sum_{\a=1}^{n+1}v^\a w^\a.
\end{align*}
The Riemannian space form $\mathbb N^{n+1}(\e)$ is the Riemannian manifold with constant curvature $\e$:
\begin{align*}
\left\{\begin{aligned}
&\text{if $\e=1$}, \quad  \mathbb S^{n+1}=\{ y \in \mathbb R^{n+2} :~ \langle y,y\rangle=1\};\\
&\text{if $\e=0$}, \quad  \mathbb R^{n+1};\\
&\text{if $\e=-1$}, \quad \mathbb H^{n+1}=\{ y\in \mathbb R^{n+1,1} :~ \langle y,y\rangle=-1,y^0>0\}.
\end{aligned}\right.
\end{align*}
Similarly, let $\mathbb N^{n,1}(\e)$ be the following Lorentzian manifold of constant curvature $\e$ (cf. \cite[P. 110,228]{Oneil1983}):
\begin{align*}
\left\{\begin{aligned}
&\text{if $\e=1$}, \quad \mathbb S^{n,1}=\{y\in \mathbb R^{n+1,1} :~ \langle y,y\rangle=1\};\\
&\text{if $\e=0$}, \quad  \mathbb R^{n,1};\\
&\text{if $\e=-1$}, \quad \mathbb H^{n,1}=\{ y\in \mathbb R^{n,2} :~ \langle y,y\rangle=-1\},
\end{aligned}\right.
\end{align*}
where $\mathbb R^{n,2}$ is the $(n+2)$-dimensional vector space endowed with the bilinear form 
\begin{align*}
\langle v,w\rangle=-v^0 w^0-v^{1}w^{1}+\sum_{\a=2}^{n+1}v^\a w^\a.
\end{align*}
Here $\mathbb S^{n,1}$ and $\mathbb H^{n,1}$ are called {\em de Sitter space} and {\em anti de Sitter space}, respectively. The upper branch of de Sitter space is defined by
\begin{align*}
\mathbb S^{n,1}_{+}=\{ y\in \mathbb R^{n+1,1}:~\langle y,y\rangle=1, ~y^0>0 \}.
\end{align*}
The Rimannian space form $\mathbb N^{n+1}(\e)$ and de Sitter space $\mathbb S^{n,1}$ can be expressed as a warped product manifold with the base manifold $\mathbb S^n$. Let $N=I \times \mathbb S^n$ be a Riemannian or Lorentzian warped product with metric 
\begin{align}\label{s2:warped-metric}
\-g=\s dr^2+\l^2(r)g_{\mathbb S^n},
\end{align}
where $g_{\mathbb S^n}=\s_{ij}d\t^id\t^j$ is the round metric of $\mathbb S^n\subset \mathbb R^{n+1}$ and
\begin{enumerate}
	\item if $N=\mathbb S^{n+1}$, then $I=[0,\pi)$, $\s=1$ and $\l(r)=\sin r$;
	\item if $N=\mathbb H^{n+1}$, then $I=[0,\infty)$, $\s=1$ and $\l(r)=\sinh r$;
	\item if $N=\mathbb S^{n,1}$, then $I=(-\infty,\infty)$, $\s=-1$ and $\l(r)=\cosh r$, see e.g. \cite[Lemma 2.1]{Scheuer19}.
\end{enumerate}

If $M$ is a compact, starshaped hypersurface in $\mathbb N^{n+1}(\e)$, then it can be represented as a smooth radial graph over $\mathbb S^n$, i.e., in the coordinates
\begin{align}\label{s2:radial-graph}
M=\{(\rho(\t),\t): \t \in\mathbb S^n\}=\{(\rho(\t(\xi)),\t(\xi)):\xi\in M\},
\end{align}
where $\rho:\mathbb S^n \ra \mathbb R^{+}$ is a smooth function. 

Similarly, it follows from \cite[Proposition 1.6.3, Remark 1.6.4]{Gerh06} that if $M$ is a compact, spacelike hypersurface in $\mathbb S_{+}^{n,1}$, then $M$ is also a smooth graph over $\mathbb S^n$ as \eqref{s2:radial-graph}, since the compact Cauchy hypersurface $\mathcal{S}_0$ can be chosen as the slice $\{y^0=0\}$ which is isometric to $\mathbb S^n$. Here we restrict the de Sitter space $\mathbb S^{n,1}$ to be its upper branch $\mathbb S^{n,1}_{+}$ and hence $\rho>0$, which guarantees that a bounded domain can be enclosed by a spacelike hypersurface and $\{y^0=0\}$. Along the hypersurface $M$ in $\mathbb S^{n,1}_{+}$, the future directed unit timelike normal is given by
\begin{align}\label{s2:unit-normal-de-Sitter}
\nu=\frac{1}{v}\( 1, -\frac{1}{\cosh^2 \rho}\s^{ik}\rho_k\),
\end{align}
where $\rho_k=\partial_{\t^k}\rho$ is the covariant derivative on $\mathbb S^n$ and
\begin{align}\label{s2:express-v}
v^2=1-\cosh^{-2}\rho \s^{ij}\rho_i \rho_j=1-\cosh^{-2}\rho|D\rho|^2.
\end{align}
The induced metric and the area element of $M$ are given by
\begin{align*}
g_{ij}=-\rho_i\rho_j+\cosh^2 \rho \s_{ij},
\end{align*}
and 
\begin{align}\label{s2:area-element}
d\mu=\sqrt{\det(g_{ij})}d\t=\cosh^n \rho v \sqrt{\det(\s_{ij})} d\t.
\end{align}
Along the variational vector field $\frac{\partial}{\partial t}X$, the spacelike hypersurfaces $M_t$ can be parametrized by 
\begin{align*}
M_t=\{(\rho(t,\t),\t): \t\in \mathbb S^n \},
\end{align*}
where $\rho(t,\cdot)$ is a smooth function defined on the unit sphere. Hence the evolution of the spacelike hypersurfaces $M_t$ can be reduced to a parabolic equation for the radial function. As long as the evolving hypersurface exists and remains spacelike, the radial function $\rho$ satisfies 
\begin{align}\label{s2:radial-evolution}
\frac{\partial \rho}{\partial t}=\langle \frac{\partial}{\partial t}X,\nu\rangle v,
\end{align}
where $v$ is given by \eqref{s2:express-v}. Notice that the velocity vector of \eqref{s2:radial-evolution} is in the direction of $\partial_r$. A direct calculation yields
\begin{align*}
\frac{\partial}{\partial t}X-\frac{\partial \rho}{\partial t}\partial_r
=&(\frac{\partial}{\partial t}X)^\top+\langle \frac{\partial}{\partial t}X,\nu\rangle(\nu-v\partial_r),
\end{align*}
and $\langle \nu-v\partial_r,\nu\rangle=0$. Therefore, the difference $\frac{\partial}{\partial t}X-\frac{\partial \rho}{\partial t}\partial_r$ is a time-dependent tangential vector field, and \eqref{s2:radial-evolution} follows by composing the original flow with the reparametrization associated with this tangential vector field.

\subsection{Duality}\label{sec:2.2}
In this subsection, we recall the duality relations in the sphere and in hyperbolic/de Sitter space via the Gauss map. For more details, we refer the readers to Gerhardt's book \cite[Chapters 9\&10]{Gerh06}. 

{\bf Duality in the sphere}. 
Let $X:M_0\ra M \subset \mathbb S^{n+1}$ be an embedding of a closed, connected, strictly convex hypersurface. Then by the Hadamard theorem in the sphere \cite{DoCarmo-Warner1970}, $M$ is embedded, homeomorphic to $\mathbb S^n$, contained in an open hemisphere and it is the boundary of a convex body $\hat{M}\subset \mathbb S^{n+1}$. Viewing $M$ as a submanifold of codimension $2$ in $\mathbb R^{n+2}$, its Gaussian formula is
\begin{align*}
X_{ij}=-g_{ij} X-h_{ij}X^\ast,
\end{align*}
where $X^\ast \in T_{X} \mathbb R^{n+2}$ represents the unit outward normal $\nu\in T_X \mathbb S^{n+1}$. Then the map 
\begin{align}\label{s2:Gauss-map}
X^\ast: M_0 \ra M^\ast \subset \mathbb S^{n+1}
\end{align} is an embedding of a closed and strictly convex hypersurface $M^\ast$. We call this map the {\em Gauss map} of $M$.

The duality in the sphere is characterized by the following theorem. 
\begin{thm}\label{s2:thm-duality-sphere}\cite[Theorem 9.2.5, 9.2.9]{Gerh06}
	Let $X:M_0\ra M \subset \mathbb S^{n+1}$ be a closed, connected, strictly convex hypersurface. Then the Gauss map $X^\ast$ given by \eqref{s2:Gauss-map} is the embedding of a closed, connected, strictly convex hypersurface $M^\ast\subset \mathbb S^{n+1}$. Viewing $M^\ast$ as a submanifold of codimension $2$ in $\mathbb R^{n+2}$, its Gaussian formula is
	\begin{align}\label{s2:Gauss-formula-polar}
	X^\ast_{ij}=-g^{\ast}_{ij} X^\ast-h^{\ast}_{ij}X,
	\end{align}  
	where $g^\ast_{ij}$ and $h^\ast_{ij}$ are the metric and the second fundamental form of the hypersurface $M^\ast \subset \mathbb S^{n+1}$ and $X=X(\xi)$ is the embedding of $M$ which also represents the unit outward normal of $M$. The second fundamental form $h^\ast_{ij}$ is defined with respect to the unit inner normal vector.

	The second fundamental forms of $M$, $M^\ast$ and the corresponding principal curvatures $\k_i,\k_i^\ast$ satisfy 
	\begin{align}\label{s2:duality-sphere}
	h_{ij}=h^\ast_{ij}=\langle X^\ast_i,X_j\rangle, \quad \k_i^\ast=\k_i^{-1},
	\end{align}
	where $\langle \cdot,\cdot\rangle$ is the inner product in $\mathbb R^{n+2}$.
\end{thm}
The hypersurface $M^\ast$ is also called the {\em polar hypersurface} of $M$, which can be equivalently defined by
$$
M^\ast=\{y\in \mathbb S^{n+1} : \sup_{x\in M}\langle x,y\rangle=0\}.
$$ 
Let $\hat{M}$ be the (closed) convex body of $M\subset \mathbb S^{n+1}$. Then the {\em dual body} of $\hat{M}$ is defined by
\begin{align}\label{s2:dual-body-sphere}
\hat{M}^\ast=\{y\in \mathbb S^{n+1} : \langle x,y\rangle\leq 0, ~\forall ~x\in \hat{M}\}.
\end{align}
By the Hadamard theorem, there exists a point $z\in \hat{M}$ such that $\hat{M}\subset \mathcal H(z)$, where $\mathcal H(z)$ is the open hemisphere centered at $z\in \mathbb S^{n+1}$. Then it follows from \cite[Corollary 9.2.10]{Gerh06} that the dual body $\hat{M}^\ast$ of $\hat{M}$ is a convex body in $\mathcal H(-z)$, where $-z\in \mathbb S^{n+1}$ is the antipodel point of $z$. Moreover, the dual body $\hat{M}^\ast$ is the convex body of $M^\ast$. Furthermore, we have $(M^\ast)^\ast=M$ and $(\hat{M}^\ast)^\ast=\hat{M}$.

{\bf Duality between hyperbolic/de Sitter space}.
Similar as the duality in the sphere, let $X:M_0 \ra M \subset \mathbb H^{n+1}$ be an embedding of a closed, connected, strictly convex hypersurface. Then the representation $X^\ast\in T_X \mathbb R^{n+1,1}$ of the unit outward normal $\nu\in T_X \mathbb H^{n+1}$ is also an embedding
\begin{align}\label{s2:Gauss-map-hyperbolic}
X^\ast: M_0 \ra M^\ast \subset \mathbb S^{n,1}
\end{align}
of a strictly convex, closed and spacelike hypersurface $M^\ast$. We also call $X^\ast$ the {\em Gauss map} of $M$. The duality in hyperbolic/de Sitter space is characterized by the following two theorems. 
\begin{thm}\label{s2:thm-duality-hyperbolic-space}\cite[Theorem 10.4.4]{Gerh06}
	Let $X:M_0\ra M \subset \mathbb H^{n+1}$ be a closed, connected, strictly convex hypersurface. Then the Gauss map $X^\ast$ given by \eqref{s2:Gauss-map-hyperbolic} is an embedding of a closed, spacelike, strictly convex hypersurface $M^\ast\subset \mathbb S^{n,1}$. Viewing $M^\ast$ as a submanifold of codimension $2$ in $\mathbb R^{n+1,1}$, its Gaussian formula is 
	\begin{align}\label{s2:Gauss-formula-polar-hyperbolic}
	X^\ast_{ij}=-g^{\ast}_{ij} X^\ast-h^{\ast}_{ij}X,
	\end{align}
	where $g^\ast_{ij}$ and $h^\ast_{ij}$ are the metric and the second fundamental form of the hypersurface $M^\ast \subset \mathbb S^{n,1}$ and $X=X(\xi)$ is the embedding of $M$ which also represents the future directed normal vector of $M^\ast$. The second fundamental form $h^\ast_{ij}$ is defined with respect to the future directed unit normal, where the time orientation is inherited from $\mathbb R^{n+1,1}$. 
	
	The second fundamental forms of $M$, $M^\ast$ and the corresponding principal curvatures $\k_i,\k_i^\ast$ satisfy 
	\begin{align}\label{s2:duality-hyperbolic}
	h_{ij}=h^\ast_{ij}=\langle X^\ast_i,X_j\rangle, \quad \k_i^\ast=\k_i^{-1}.
	\end{align}
	where $\langle \cdot,\cdot\rangle$ is the inner product in $\mathbb R^{n+1,1}$. 
\end{thm}
The hypersurface $M^\ast$ is also called the {\em polar hypersurface} of $M$, which can be also equivalently defined by
$$
M^\ast=\{y\in \mathbb S^{n,1} : \sup_{x\in M}\langle x,y\rangle=0\}.
$$
Let $\hat{M}$ be the (closed) convex body of $M\subset \mathbb H^{n+1}$, which is characterized by (see \cite[Lemma 10.4.1]{Gerh06})
\begin{align}\label{s2:convex-body-hyperbolic}
\hat{M}=\{y\in \mathbb H^{n+1}: \langle y,X^\ast \rangle\leq 0,~\forall X\in M \}.
\end{align}
In the hyperboloid model of the hyperbolic space, the point $e_0=(1,0,\cdots,0)\in \mathbb R^{n+1,1}$ is called the {\em Beltrami point}. The Hadamard theorem in hyperbolic space (\cite[Theorem 10.3.1]{Gerh06}) states that any strictly convex hypersurface $M\subset \mathbb H^{n+1}$ bounds a strictly convex body $\hat{M}$ of the hyperbolic space. Due to the homogeneity of the hyperbolic space, any interior point $z$ in $\hat{M}$ may act as a Beltrami point by using a Lorenz boost $\mathbb O_z\in O(n+1,1)$ to bring $z$ to $e_0$. By \cite[Theorem 10.4.9]{Gerh06}, the Gauss maps provide a bijection between the closed, strictly convex and connected hypersurfaces $M\subset \mathbb H^{n+1}$ having the Beltrami point in the interior of their convex bodies and the spacelike, closed, connected and strictly convex hypersurfaces $M^\ast\subset \mathbb S^{n,1}_{+}$, where $\mathbb S^{n,1}_{+}$ is the upper branch of the de Sitter space. Therefore, the polar hypersurface $M^\ast$ of $M$ is contained in $\mathbb S^{n,1}_{+}$. Moreover, the dual body of $\hat{M}$ can be defined by
\begin{align}\label{s2:dual-body-deSitter}
\hat{M}^\ast=\{y\in \mathbb S^{n,1} : \langle x,y\rangle\leq 0, ~\forall x\in \hat{M}\}\subset \mathbb S^{n,1}_{+},
\end{align}
which is the bounded domain enclosed by the slice $\{y^0=0\}$ and the hypersurface $M^\ast$.  

The reverse direction starting from a closed, spacelike, strictly convex hypersurface in $\mathbb S^{n,1}_{+}$ will be stated as follows:
\begin{thm}\label{s2:thm-duality-de-Sitter-space}\cite[Theorem 10.4.5]{Gerh06}
	Let $X^\ast :M_0\ra M^\ast \subset \mathbb S_{+}^{n,1}$ be a closed, connected, spacelike and strictly convex hypersurface. Viewing $M^\ast$ as a submanifold of codimension $2$ in $\mathbb R^{n+1,1}$, its Gaussian formula is 
	\begin{align}\label{s2:Gauss-formula-polar-deSitter}
	X^\ast_{ij}=-g^{\ast}_{ij} X^\ast+h^{\ast}_{ij}X,
	\end{align}
	where $X^\ast=X^\ast(\xi)$ is the embedding, $X$ the future directed unit normal, and $g^\ast_{ij}$ and $h^\ast_{ij}$ the induced metric and the second fundamental form of the hypersurface $M^\ast \subset \mathbb S^{n,1}$. We define the Gauss map as $X=X(\xi)$ 
	\begin{align*}
	X: M_0 \ra M \subset \mathbb H^{n+1}.
	\end{align*} 
	Then the Gauss map is the embedding of a closed, connected, strictly convex hypersurface $M$ in $\mathbb H^{n+1}$. Viewing $M$ as a submanifold of codimension $2$ in $\mathbb R^{n+1,1}$, then its Gaussian formula is 
	\begin{align*}
	X_{ij} =g_{ij} X-h_{ij}X^\ast.
	\end{align*}
	where $g_{ij}$, $h_{ij}$ be the induced metric and the second fundamental form of $M$. The second fundamental forms of $M$, $M^\ast$ and the corresponding principal curvatures $\k_i,\k_i^\ast$ satisfy 
	\begin{align}\label{s2:duality-deSitter}
	h_{ij} =h^\ast_{ij}=\langle X_i,X^\ast_j\rangle, \quad 	\k^\ast_i=\k_i^{-1},
	\end{align}	
\end{thm} 
Combining Theorems \ref{s2:thm-duality-hyperbolic-space} with \ref{s2:thm-duality-de-Sitter-space}, we conclude that $(M^\ast)^\ast=M$ and $(\hat{M}^\ast)^\ast=\hat{M}$.

\begin{rem}
	For a convex body $K$ in Euclidean space, its dual body $K_z^\ast$ depends on the choice of the interior point $z$ in $K$. For a strictly convex body $K$ in $\mathbb S^{n+1}$ or $\mathbb H^{n+1}$, or a spacelike and strictly convex domain $K$ in $\mathbb S^{n,1}_{+}$, its dual body $K^\ast$ in $\mathbb S^{n+1}$, $\mathbb S^{n,1}_{+}$ or $\mathbb H^{n+1}$ of $K$ is defined by \eqref{s2:dual-body-sphere}, \eqref{s2:dual-body-deSitter} and \eqref{s2:convex-body-hyperbolic}, respectively. This definition is independent of the choice of the interior point in $K$.
\end{rem} 

\subsection{Dual flows}
The dual flows for (pure) curvature flows in the sphere was first introduced by Gerhardt \cite{Ge15}, and later it was also used by Yu \cite{Yu2016} for curvature flows in hyperbolic space. For general curvature flows, the deduction of the dual flows can be found in \cite[(4.3),(4.4)]{Bryan-Ivaki-Scheuer2020}. 

Let $X_0: M_0^n\ra (N^{n+1},\-g)$ be a smooth embedding such that $M=X_0(M_0)$ is a closed, strictly convex hypersurface, where $N$ is one of $\mathbb S^{n+1}$, $\mathbb H^{n+1}$ and $\mathbb S^{n,1}_{+}$. We consider the smooth family of strictly convex embeddings $X:M_0 \times [0,T)\ra N$ satisfying
\begin{align}\label{s2:curvature-flow}
\left\{
\begin{aligned}
\frac{\partial}{\partial t}X=&-\s \mathcal{F}\nu, \\
X(\cdot,0)=&X_0(\cdot).
\end{aligned}\right.
\end{align}
where $\nu$ is the unit normal of $M_t=X(M_0,t)$ and $\mathcal{F}$ is a smooth and strictly monotone function of the principal curvatures of $M_t$. Here the signature $\s=\-g(\nu,\nu)=\langle \nu,\nu \rangle$, where $\langle \cdot,\cdot \rangle$ represents the Euclidean or the Minkowski inner product for the flows in $\mathbb S^{n+1}$ or $\mathbb H^{n+1}$, $\mathbb S^{n,1}$, respectively. In all three cases, the pair $X$, $X^\ast=\nu$ satisfies $\langle X,X^\ast \rangle=0$. It follows that
\begin{align*}
\langle\frac{\partial}{\partial t}X^\ast, X \rangle=-\langle \frac{\partial}{\partial t}X, X^\ast \rangle=-\langle -\s \mathcal{F}X^\ast, X^\ast \rangle =\mathcal{F}.
\end{align*}
Due to $\langle X^\ast,X^\ast_{i} \rangle=0$, the Weingarten equation $X^\ast_i=h_i^k X_k$ (see also \cite[Lemma 9.2.4, Lemma 10.4.3]{Gerh06}) imply that
\begin{align*}
0=\langle \frac{\partial}{\partial t}X^\ast, X^\ast_i \rangle=\langle \frac{\partial}{\partial t}X^\ast, h_i^k X_k\rangle
 =-h_i^k \langle X^\ast, (-\s \mathcal{F} X^\ast)_k\rangle 
 =\s h_i^k \langle X^\ast,  \mathcal{F}_k X^\ast \rangle
 =h_i^k \mathcal{F}_k,
\end{align*}
where $\mathcal{F}_k=\frac{\partial \mathcal{F}}{\partial \xi^k}$. Note that $X=(X^\ast)^\ast=\nu^\ast$ and $X^\ast_{i}=h_i^k X_k$ span $T_{X^\ast}\mathbb S^{n+1}$, $T_{X^\ast}\mathbb H^{n+1}$ or $T_{X^\ast}\mathbb S^{n,1}$ respectively, we obtain
\begin{align*}
\frac{\partial}{\partial t}X^\ast=\langle X,X\rangle \mathcal{F}X+ h_i^k \mathcal{F}_{k} (g^\ast)^{il}X^\ast_l=\s^\ast \mathcal{F} \nu^\ast+ h_i^k (g^\ast)^{il} \mathcal{F}_k X^\ast_{l}=\s^\ast \mathcal{F} \nu^\ast+ (b^\ast)^{kl} \mathcal{F}_k X^\ast_{l},
\end{align*}
where $\s^{\ast}=\langle X,X\rangle$, $((b^{\ast})_i^j)$ is the inverse matrix of the Weingarten matrix $\mathcal{W}^\ast=((h^{\ast})_i^j)$ of $M^\ast$ and $h_i^k=(b^{\ast})_i^k$ due to \eqref{s2:duality-sphere}, \eqref{s2:duality-hyperbolic} and \eqref{s2:duality-deSitter}. Let us define  
\begin{align}\label{s2:dual-function}
\mathcal{F}^\ast(\mathcal{W}^\ast):=-\mathcal{F}(\mathcal{W}).
\end{align}
Let $N^\ast$ be dual ambient space of $N$, that is, $\mathbb S^{n+1}$, $\mathbb S^{n,1}$ and $\mathbb H^{n+1}$, respectively. Then the flow of the polar hypersurfaces is a smooth family of strictly convex embeddings $X^\ast:M_0\times [0,T)\ra N^\ast$ satisfying
\begin{align}\label{s2:dual-flow}
\left\{
\begin{aligned}
\frac{\partial}{\partial t}X^\ast=&-\s^{\ast} \mathcal{F}^\ast \nu^\ast-(b^\ast)^{kl}\mathcal{F}^\ast_{k}X^\ast_l, \\
X^\ast(\cdot,0)=&X^\ast_0(\cdot).
\end{aligned}\right.
\end{align}
where $\mathcal{F}^{\ast}$ is evaluated at the Weingarten matrix $\mathcal{W}^\ast$ of $M_t^\ast=X^\ast(M_0,t)$, and $X^\ast_0:M_0 \ra N^\ast$ is the smooth embedding of $M^\ast=X_0^{\ast}(M_0)$ which is determined by the Gauss map of the hypersurface $M$. Therefore, a flow of the form \eqref{s2:curvature-flow} in the ambient spaces $\mathbb S^{n+1}$, $\mathbb H^{n+1}$, $\mathbb S^{n,1}$ has a dual flow of the form \eqref{s2:dual-flow} in the ambient spaces $\mathbb S^{n+1}$, $\mathbb S^{n,1}$, $\mathbb H^{n+1}$, respectively.

\section{Quermassintegrals}\label{sec:3}
Let $K$ be a bounded domain with smooth boundary $\partial K$ in Riemannian space forms $\mathbb N^{n+1}(\e)$ with constant curvature $\e\in \{-1,0,1\}$. Then the quermassintegrals $W_k(K)$ can be expressed by a linear combination of the curvature integrals and of the enclosed volume as follows:
\begin{equation}\label{s1:quermassintegral-def}
\begin{split}
&W_0(K)=\operatorname{Vol}(K), \quad  W_1(K)=\frac{1}{n+1}\operatorname{Area}(\partial K), \\
&W_{k+1}(K)=\frac{1}{n+1}\int_{\partial K}E_{k}d\mu+\e\frac{k}{n+2-k}W_{k-1}(K), \quad k=1,\cdots,n,
\end{split}
\end{equation}
where $E_k$ is the normalized $k$th mean curvature and $d\mu$ is the area element of $\partial K$. 

When the ambient space is the upper branch of the de Sitter space, then the volume of the region $K$ enclosed by the slice $\{y^0=0\}$ and a spacelike graph $\partial K=\{(\rho(\t),\t):\t \in \mathbb S^n\}$ can be expressed by (see \cite[Section 4]{Makowski2013})
\begin{align*}
\Vol(K)=\int_{\mathbb S^n} \int_0^{\rho(\t)} \sqrt{|\det(\-g_{ij}(r,\t))|} dr d\t=\int_{\mathbb S^n}\int_0^{\rho(\t)}\cosh^n r \sqrt{\det(\s_{ij})}dr d\t,
\end{align*}
where $\-g=-dr^2+\cosh^2r g_{\mathbb S^n}$. As the hypersurface $\partial K$ is spacelike, the surface area of $\partial K$ is defined as usual by its induced Riemannian metric. For simplicity, we also call $K$ to be a {\em spacelike domain} in $\mathbb S^{n,1}_{+}$ if $\partial K$ is a spacelike hypersurface. It is natural to define the quermassintegrals for a smooth bounded domain $K$ with spacelike boundary $\partial K$ in the upper branch of the de Sitter space. In this direction, Scheuer \cite{Scheuer19} defined the second quermassintegral as 
\begin{align*}
W_2(K)=\frac{1}{n+1}\int_{\partial K}E_{1}d\mu-\frac{1}{n+1}\Vol(K).
\end{align*}
Inspired by the expression \eqref{s1:quermassintegral-def} of quermassintegrals in the Riemannian space forms, the quermassintegrals for a spacelike domain $K$ in $\mathbb S_+^{n,1}$ can be defined inductively by
\begin{equation}\label{s1:quermassintegral-def-de-Sitter}
\begin{split}
&W_0(K)=\operatorname{Vol}(K), \quad  W_1(K)=\frac{1}{n+1}\operatorname{Area}(\partial K), \\
&W_{k+1}(K)=\frac{1}{n+1}\int_{\partial K}E_{k}d\mu-\frac{k}{n+2-k}W_{k-1}(K), \quad k=1,\cdots,n.
\end{split}
\end{equation}
In Riemannian space forms $\mathbb N^{n+1}(\e)$, the definition \eqref{s1:quermassintegral-def} coincides with the measure of totally geodesic $k$-dimensional subspaces in $\mathbb N^{n+1}(\e)$ which intersect $K$, which is the definition of quermassintegrals from integral geometry, see \cite{Sant2004} or \cite[Proposition 7]{Solanes06}. It would be an interesting question whether one can establish the integral geometric theory for smooth spacelike domains in $\mathbb S_+^{n,1}$. For example, the Crofton formulas and Cauchy formulas have been extended to Lorentzian constant curvature space, see \cite{Solanes-Teufel05}.
\begin{rem}
	It is tempting to ask that whether or not there is a similar definition for the quermassintegrals in the other Lorentzian space forms, i.e., the Minkowski space $\mathbb R^{n,1}$ or the anti de Sitter space $\mathbb H^{n,1}$. However, we are not able to define the surface area of the complete spacelike hypersurfaces in $\mathbb R^{n,1}$ and $\mathbb H^{n,1}$. For example, the slice $\{y^0=0\}$ in $\mathbb R^{n,1}$ or $\mathbb H^{n,1}$ is a complete noncompact spacelike hypersurface, which it is isometric to $\mathbb R^n$ or $\mathbb H^n$, respectively. 
\end{rem}

The $k$th {\em mean radius} of the smooth bounded domain $K$ in $\mathbb N^{n+1}(\e)$ or a spacelike domain $K$ in $\mathbb S^{n,1}_{+}$ can be similarly defined by
\begin{align}\label{s3:def-mean-radius-spaceforms}
\zeta_k(K)=f_k^{-1}(W_k(K)),
\end{align}
where $f_k:[0,\infty)\ra \mathbb R_{+}$ is a smooth function determined by $f_k(s)=W_k(B_s)$ in $\mathbb N^{n+1}(\e)$, where $B_s$ is the geodesic ball of radius $s$; or $f_k(s)=W_k(\{0\leq r\leq s\})$ in $\mathbb S^{n,1}_{+}$, where $\{0\leq r\leq s\}$ denotes the region enclosed by the slice $\{y^0=0\}$ and the slice $\{r=s\}$ in $\mathbb S_+^{n,1}$. In all above cases, it is easy to check that $f_k$ is monotone increasing, and $f_k^{-1}$ is the inverse of $f_k$. 

The key feature of quermassintegrals is the following elegant variational formulas in Riemannian space forms and in de Sitter space.
\begin{lem}
	Assume that one of the following holds:
	\begin{enumerate}[(i)]
		\item Let $K_t$ be a family of smooth bounded domains in $\mathbb N^{n+1}(\e)$ such that $\partial K_t$ is a family of smooth closed hypersurfaces. 
		\item Let $K_t$ be a family of smooth bounded domains in $\mathbb S^{n,1}_{+}$ such that $\partial K_t$ is a family of smooth closed spacelike hypersurfaces.
    \end{enumerate}
    Let $\frac{\partial}{\partial t}X$ be the variational vector field of $\partial K_t$. Then 
	\begin{align}\label{s3:variational-formula}
	\frac{d}{dt}W_k(K_t)=\frac{n+1-k}{n+1}\int_{\partial K_t} \langle \frac{\partial}{\partial t}X,\nu \rangle E_k d\mu_t, \quad k=0,1,\cdots,n.
	\end{align}	
\end{lem}
\begin{proof}
	For simplicity, we denote by $f=\langle\frac{\partial}{\partial t}X,\nu \rangle$, where $\langle \cdot,\cdot \rangle$ denotes the Euclidean inner product for $\mathbb S^{n+1}$, $\mathbb R^{n+1}$ and the Minkowski inner product for $\mathbb H^{n+1}$,  $\mathbb S^{n,1}_{+}$, respectively.
	
	(i) By \cite[Theorem A]{Reilly73}, the curvature integral of $\partial K_t$ evolves by
	\begin{align}\label{s2:curvature-integral}
	\frac{d}{dt}\int_{\partial K_t} E_k d\mu_t = \int_{\partial K_t} \( (n-k)E_{k+1}-\e kE_{k-1}   \)f d\mu_t, \quad k=1,\cdots,n.
	\end{align} 
	The volume of $K_t$ evolves by (see e.g. \cite{Barbosa-doCarmo-Eschenburg88})
	\begin{align}\label{s2:volume}
	\frac{d}{dt}\operatorname{Vol}(K_t)=\int_{\partial K_t} f d\mu_t.
	\end{align}
	Now the variational formula \eqref{s3:variational-formula} can be proved by induction (see \cite[Proposition 3.1]{Wang-Xia14}). In view of \eqref{s2:curvature-integral} and \eqref{s2:volume}, it is true for $k=0, 1$. Assume that it is true for $k-1$, that is,
	\begin{align}\label{s3:induction-assumption}
	\frac{d}{dt}W_{k-1}(K_t)=\frac{n+1-(k-1)}{n+1}\int_{\partial K_t} E_{k-1}f d\mu_t,
	\end{align}
	then by using \eqref{s1:quermassintegral-def}, \eqref{s2:curvature-integral} and the inductive assumption \eqref{s3:induction-assumption}, we can compute that
	\begin{align*}
	\frac{d}{dt}W_{k+1}(K_t)=&\frac{1}{n+1}\frac{d}{dt}\int_{\partial K_t}E_k d\mu_t+\e \frac{k}{n+2-k}\frac{d}{dt}W_{k-1}(K_t)\\
	=&\frac{1}{n+1}\int_{\partial K_t}\((n-k)E_{k+1}-\e kE_{k-1}\) f d\mu_t+\e \frac{k}{n+1}\int_{\partial K_t} E_{k-1}f d\mu_t \\
	=&\frac{n+1-(k+1)}{n+1}\int_{\partial K_t} E_{k+1}f d\mu_t.
	\end{align*}
		
	(ii) When the ambient space is $\mathbb S_{+}^{n,1}$, along the variational vector field $\frac{\partial}{\partial t}X$, the volume of the spacelike domain $K_t$ evolves by
	\begin{align*}
	\frac{d}{dt}\Vol(K_t)=&\frac{d}{dt}\int_{\mathbb S^n}\int_0^{\rho(t,\t)}	\cosh^n r \sqrt{\det(\s_{ij})}dr d\t\\
	=&\int_{\mathbb S^n}\cosh^n \rho(t,\t) \frac{\partial\rho}{\partial t} \sqrt{\det(\s_{ij})} d\t \\
	=&\int_{\mathbb S^n}\cosh^n \rho(t,\t) f v \sqrt{\det(\s_{ij})} d\t \\
	=&\int_{\partial K_t} f d\mu,
	\end{align*}
	where we used \eqref{s2:radial-evolution} and \eqref{s2:area-element}. By using the similar calculation as in \cite[Theorem A]{Reilly73}, the curvature integral of $\partial K_t$ evolves by (see e.g. \cite[Proposition 2]{Solanes-Teufel05})
	\begin{align}\label{s2:curvature-integral-de-Sitter}
	\frac{d}{dt}\int_{\partial K_t} E_k d\mu_t = \int_{\partial K_t} \( (n-k)E_{k+1}+kE_{k-1}   \)f d\mu_t, \quad k=1,\cdots,n.
	\end{align} 
	Then the remaining proof is the same as (i). 
\end{proof}

For any smooth bounded domains in $\mathbb N^{n+1}(\e)$, the quermassintegral $W_k$ can be expressed as follows (see e.g. \cite[Corollary 8]{Solanes06} for the case in $\mathbb H^{n+1}$):
\begin{itemize} 
	\item for $1\leq m\leq n$ and $m$ being even,
	\begin{align}\label{s3:quermassintegral-expression-even}
	W_m(K)=&\frac{1}{n+1}\sum_{i=0}^{\frac{m}{2}-1}\e^i \frac{(m-1)!!(n+1-m)!!}{(m-1-2i)!!(n+1-m+2i)!!}\int_{\partial K}E_{m-1-2i}d\mu \nonumber\\
	&+\e^{\frac{m}{2}}\frac{(m-1)!!(n+1-m)!!}{(n+1)!!}\Vol(K);
	\end{align}
	\item for $1\leq m\leq n$ and $m$ being odd,
	\begin{align}\label{s3:quermassintegral-expression-odd}
	W_m(K)=\frac{1}{n+1}\sum_{i=0}^{\frac{m-1}{2}}\e^i \frac{(m-1)!!(n+1-m)!!}{(m-1-2i)!!(n+1-m+2i)!!}\int_{\partial K}E_{m-1-2i}d\mu.
	\end{align}
\end{itemize}
Here the notation $k!!$ means the product of all odd(even) integers up to odd(even) $k$. In view of \eqref{s1:quermassintegral-def-de-Sitter}, the quermassintegrals in $\mathbb S^{n,1}_{+}$ coincide with those in $\mathbb H^{n+1}$, i.e.,
\begin{itemize} 
	\item for $1\leq m\leq n$ and $m$ being even,
	\begin{align}\label{s3:quermassintegral-expression-even-de-Sitter}
	W_m(K)=&\frac{1}{n+1}\sum_{i=0}^{\frac{m}{2}-1}(-1)^i \frac{(m-1)!!(n+1-m)!!}{(m-1-2i)!!(n+1-m+2i)!!}\int_{\partial K}E_{m-1-2i}d\mu \nonumber\\
	&+(-1)^{\frac{m}{2}}\frac{(m-1)!!(n+1-m)!!}{(n+1)!!}\Vol(K);
	\end{align}
	\item for $1\leq m\leq n$ and $m$ being odd,
	\begin{align}\label{s3:quermassintegral-expression-odd-de-Sitter}
	W_m(K)=\frac{1}{n+1}\sum_{i=0}^{\frac{m-1}{2}}(-1)^i \frac{(m-1)!!(n+1-m)!!}{(m-1-2i)!!(n+1-m+2i)!!}\int_{\partial K}E_{m-1-2i}d\mu.
	\end{align}
\end{itemize}

\section{Proofs of Theorem \ref{thm-sphere} and Theorem \ref{thm-hyperbolic}}\label{sec:4}
Along the flow \eqref{s2:curvature-flow}, we have the following identities for convex bodies and their dual bodies.
\begin{lem}\label{s4:lem-identity}
	Assume that one of the following holds:
	\begin{enumerate}[(i)]
		\item Let $K_t$ be a family of smooth bounded domain with strictly convex boundary $\partial K_t$ in $\mathbb S^{n+1}$, and $K_t^\ast$ the dual body of $K_t$ in $\mathbb S^{n+1}$;
		\item Let $K_t$ be a family of smooth bounded domain with strictly convex boundary $\partial K_t$ in $\mathbb H^{n+1}$, and $K_t^\ast$ the dual body of $K_t$ in $\mathbb S^{n,1}_{+}$;
		\item Let $K_t$ be a family of smooth bounded domain with strictly convex and spacelike boundary $\partial K_t$ in $\mathbb S^{n,1}_{+}$, and $K_t^\ast$ the dual body of $K_t$ in $\mathbb H^{n+1}$.
\end{enumerate}	
	    Assume that $X:M_0 \times [0,T) \ra N$ is a family of smooth embeddings satisfying \eqref{s2:curvature-flow} such that $X_t(M_0)=\partial K_t$, then there holds
	\begin{align}\label{s4:evol-identity}
	\frac{d}{dt}\left[ (k+1)W_k(K_t)+(n+1-k)W_{n-k}(K_t^\ast)\right]=0, \quad k=0,\cdots,n.
	\end{align}
\end{lem}
\begin{proof}
	Along the flow \eqref{s2:curvature-flow} in $N$, i.e.,
	\begin{align*}
	\frac{\partial}{\partial t}X=-\s \mathcal{F}\nu,
	\end{align*}
	it follows from \eqref{s3:variational-formula} that
	\begin{align}\label{s4:quermassintegral-direct-flow}
	\frac{d}{dt}W_k(K_t)=&\frac{n+1-k}{n+1}\int_{\partial K_t} E_k(\k) \langle \frac{\partial}{\partial t}X,\nu\rangle d\mu_t\nonumber \\ 
	=&-\frac{n+1-k}{n+1}\int_{M_t} E_k(\k) \mathcal{F} d\mu_t,
	\end{align}
	where $\langle \cdot,\cdot\rangle$ denotes the Euclidean inner product for $\mathbb S^{n+1}$ and the Minkowski inner product for $\mathbb H^{n+1}$ and $\mathbb S^{n,1}_{+}$, respectively.
	 
	Corresponding to the flow \eqref{s2:curvature-flow}, the dual flow \eqref{s2:dual-flow} in the dual ambient space $N^\ast$ is 
	\begin{align}
	\frac{\partial}{\partial t}X^\ast=-\s^\ast \mathcal{F}^\ast \nu^\ast-(b^\ast)^{kl}\mathcal{F}^\ast_k X^\ast_l.
	\end{align}
	It also follows from \eqref{s3:variational-formula} that
	\begin{align}\label{s4:quermassintegral-dual-flow}
	\frac{d}{dt}W_{n-k}(K^\ast_t)=&\frac{n+1-(n-k)}{n+1}\int_{\partial K^\ast_t} E_{n-k}(\k^\ast) \langle \frac{\partial}{\partial t}X^\ast, \nu^\ast \rangle d\mu^\ast_t \nonumber\\
	=&-\frac{n+1-(n-k)}{n+1}\int_{\partial K^\ast_t} E_{n-k}(\k^\ast) \mathcal{F}^\ast d\mu^\ast_t \nonumber\\
	=&\frac{k+1}{n+1}\int_{\partial K_t} E_k(\k) \mathcal{F} d\mu_t,
	\end{align}
	where we used \eqref{s2:dual-function} and 
	\begin{align*}
	E_{n-k}(\k^\ast)=\frac{E_k(\k)}{E_{n}(\k)}, \quad d\mu^\ast_t=E_n(\k) d\mu_t
	\end{align*}
	due to the duality relations \eqref{s2:duality-sphere}, \eqref{s2:duality-hyperbolic} and \eqref{s2:duality-deSitter}. Therefore, the identity \eqref{s4:evol-identity} follows from \eqref{s4:quermassintegral-direct-flow} and \eqref{s4:quermassintegral-dual-flow}.	
\end{proof}

\begin{proof}[Proof of Theorem \ref{thm-sphere}]
	Let $M=X_0(M_0)$ be a strictly convex hypersurface in $\mathbb S^{n+1}$ and $\hat{M}$ the convex body of $M$. By Lemma \ref{s4:lem-identity}, the functional $(k+1)W_k(\hat{M}_t)+(n+1-k)W_{n-k}(\hat{M}_t^\ast)$ remains invariant for any flow preserving the strict convexity. To determine the value of this functional, we consider the harmonic mean curvature flow in the sphere, i.e., 
	\begin{align*}
	\frac{\partial}{\partial t}X=-\frac{E_{n}}{E_{n-1}}\nu.
	\end{align*}
	Then by the convergence result of contracting curvature flows in the sphere due to Gerhardt \cite{Ge15}, we conclude that $\hat{M}_t$ shrinks to a fixed point $z\in \mathbb S^{n+1}$, while the dual body $\hat{M}^\ast_t$ expands to the half ball $\mathcal{H}(-z)$ as $t\ra T^\ast$. Here $T^\ast<\infty$ is both the maximal existence time of the contracting curvature flow and its dual flow. Moreover, the flow hypersurfaces $M_t$ and its polar hypersurfaces $M^\ast_t$ are strictly convex for all $t\in [0,T^\ast)$.
	
	Since $\int_{\partial B_{r}} E_m d\mu=\omega_n \sin^{n-m}r \cos^m r$, we have 
	\begin{align*}
	\lim_{r\ra 0^+}\int_{\partial B_{r}} E_m d\mu=\left\{\begin{aligned}
	&0, \quad  &\text{if $0\leq m<n$};\\
	&\omega_n, \quad &\text{if $m=n$},
	\end{aligned}\right.
	\end{align*}
	and
	\begin{align*}
	\lim_{r\ra (\frac{\pi}{2})^-}\int_{\partial B_{r}} E_m d\mu=\left\{\begin{aligned}
	&0, \quad  &\text{if $0<m\leq n$};\\
	&\omega_n, \quad &\text{if $m=0$}.
	\end{aligned}\right.
	\end{align*}
	For $1\leq m\leq n$, in view of \eqref{s3:quermassintegral-expression-even}-- \eqref{s3:quermassintegral-expression-odd}, we have
	\begin{align*}
	\lim_{r\ra 0^+}W_m(B_r)=0, 
	\end{align*}
	and
	\begin{align*}
	\lim_{r\ra (\frac{\pi}{2})^-}W_m(B_r)=\left\{\begin{aligned} 
	&\frac{\omega_n}{n+1}\frac{(m-1)!!(n+1-m)!!}{n!!},  \quad \text{if $m$ is odd};\\
	&\frac{(m-1)!!(n+1-m)!!}{(n+1)!!}\Vol(B_{\pi/2}),  \quad \text{if $m$ is even}.
	\end{aligned}  \right.
	\end{align*}
	Using the fact that $\Vol(B_{\pi/2})=\frac{1}{2}|\mathbb S^{n+1}|$, and $\omega_i=|\mathbb S^i|=\frac{(i+1)\pi^{(i+1)/2}}{\G((i+1)/2+1)}$, where $\G(\cdot)$ is the Gamma function. A direct calculation gives
	\begin{align*}
	\frac{\omega_{n+1}}{\omega_n}=\frac{(n+2)\pi^{(n+2)/2}}{\G(\frac{n+2}{2}+1)} \cdot \frac{\G(\frac{n+1}{2}+1)}{(n+1)\pi^{(n+1)/2}} = \left\{ \begin{aligned}
	&\frac{2(n-1)!!}{n!!}, \quad   \text{if $n$ is odd},\\
	&\frac{\pi(n-1)!!}{n!!}, \quad   \text{if $n$ is even}.
	\end{aligned}\right.
	\end{align*}
	Then we have
	\begin{align*}
	\lim_{r\ra (\frac{\pi}{2})^-}W_m(B_r)=\left\{\begin{aligned} 
	&\frac{\pi}{2}\frac{\omega_n}{n+1} \frac{(m-1)!!(n+1-m)!!}{n!!},  \quad &\text{if $n$ is even and $m$ is even};\\
	&\frac{\omega_n}{n+1}\frac{(m-1)!!(n+1-m)!!}{n!!},  \quad &\text{otherwise}.
	\end{aligned}  \right.
	\end{align*}
	Therefore, we get
	\begin{align*}
	 &(k+1)W_k(\hat{M})+(n+1-k)W_{n-k}(\hat{M}^\ast)\\
	=&\lim_{r\ra 0^+}\left[(k+1)W_k(B_r)+(n+1-k)W_{n-k}(B_{\frac{\pi}{2}-r})\right]\\
	=&(n+1-k)\lim_{r \ra (\frac{\pi}{2})^-}W_{n-k}(B_r)\\
	=&\left\{ \begin{aligned}
		&\frac{\pi}{2}\frac{\omega_n}{n+1} \frac{(n-k+1)!!(k+1)!!}{n!!},  \quad &\text{if $n$ is even and $k$ is even};\\
	&\frac{\omega_n}{n+1}\frac{(n-k+1)!!(k+1)!!}{n!!},  \quad &\text{otherwise}.
	\end{aligned}  \right.
	\end{align*}
	In particular, for any pairs $(\hat{M},\hat{M}^\ast)$ of strictly convex bodies and their dual bodies in $\mathbb S^{n+1}$, the following identities hold
	\begin{align*}
	(k+1)W_k(\hat{M})+(n+1-k)W_{n-k}(\hat{M}^\ast)=C_{n,k}, \quad k=0,1,\cdots,n,
	\end{align*} 
	where $C_{n,k}$ are constants depending only on $n,k$. If the geodesic ball $B_r$ satisfies $W_k(B_r)=W_k(\hat{M})$, then $W_{n-k}(B_{\frac{\pi}{2}-r})=W_{n-k}(\hat{M}^\ast)$.  Therefore, we get 
	\begin{align*}
	\zeta_k(\hat{M})+\zeta_{n-k}(\hat{M}^\ast)=\frac{\pi}{2}, \quad 0\leq k \leq n,
	\end{align*}
	which is equivalent to \eqref{s1:identity-sphere}. This completes the proof of Theorem \ref{thm-sphere}.
\end{proof}

\begin{proof}[Proof of Theorem \ref{thm-hyperbolic}]
    Let the initial hypersurface $M=X_0(M_0)$ be strictly convex in $\mathbb H^{n+1}$ and $\hat{M}$ the convex body of $M$. Similar as above, we consider the harmonic mean curvature flow in $\mathbb H^{n+1}$:
    \begin{align*}
    \frac{\partial}{\partial t} X=-\frac{E_{n}}{E_{n-1}}\nu,
    \end{align*}
    By the convergence result of Yu \cite[Theorem 1.2]{Yu2016}, the contracting hypersurface $M_t=X_t(M_0)$ in $\mathbb H^{n+1}$ shrinks to a point $x_0$ as $t\ra T^\ast$, where $T^\ast<\infty$ is the maximal existence time of this flow. By using a Lorenz boost, we may assume this point $x_0$ is the Beltrami point. Then the polar hypersurface $M_t^\ast$ of the dual flow is contained in $\mathbb S^{n,1}_{+}$ and converge to the coordinate slice $\{r=0\}$ as $t\ra T^\ast$. Moreover, the flow hypersurface $M_t$ in $\mathbb H^{n+1}$ is strictly convex, while its polar hypersurface $M^\ast_t$ in $\mathbb S^{n,1}_{+}$ is strictly convex and spacelike for all $t\in [0,T^\ast)$. Let $\hat{M}_t$ be the convex body of $M_t$ in $\mathbb H^{n+1}$, and $\hat{M}_t^\ast$ its dual body in $\mathbb S^{n,1}_{+}$, respectively.
    
    By Lemma \ref{s4:lem-identity}, the functional $(k+1)W_k(\hat{M}_t)-(n+1-k)W_{n-k}(\hat{M}_t^\ast)$ remains invariant for any flow which preserves the strict convexity. To determine the value of this functional, we first observe that 
    \begin{align*}
    \int_{\partial B_r}E_m(\k)d\mu=&\omega_n \sinh^{n-m}r \cosh^m r, \\ \int_{\{r=r^\ast\}}E_m(\k^\ast) d\mu^\ast=&\omega_n \cosh^{n-m} r^\ast \sinh^m r^\ast.
    \end{align*}
    It follows that 
    \begin{align*}
    \lim_{r\ra 0^+}\int_{\partial B_{r}} E_m(\k) d\mu=\left\{\begin{aligned}
    &0, \quad  &\text{if $0\leq m<n$};\\
    &\omega_n, \quad &\text{if $m=n$},
    \end{aligned}\right.
    \end{align*}
    and
    \begin{align*}
    \lim_{r^\ast \ra 0^{+}}\int_{\{r=r^\ast\}} E_m(\k^\ast) d\mu^\ast=\left\{\begin{aligned}
    &0, \quad  &\text{if $0<m\leq n$};\\
    &\omega_n, \quad &\text{if $m=0$}.
    \end{aligned}\right.
    \end{align*}
    For $1\leq m\leq n$, in view of \eqref{s3:quermassintegral-expression-even}--\eqref{s3:quermassintegral-expression-odd-de-Sitter}, we have
    \begin{align*}
    \lim_{r\ra 0^+}W_m(B_r)=0, 
    \end{align*}
    and
    \begin{align*}
    \lim_{r^\ast\ra 0^+}W_m(\{r=r^\ast\})=\left\{\begin{aligned} 
    &(-1)^\frac{m-1}{2}\frac{\omega_n}{n+1}\frac{(m-1)!!(n+1-m)!!}{n!!},  \quad &\text{if $m$ is odd};\\
    &0,  \quad &\text{if $m$ is even}.
    \end{aligned}  \right.
    \end{align*}
    Therefore, we get
    \begin{align*}
    &(k+1)W_k(\hat{M})+(n+1-k)W_{n-k}(\hat{M}^\ast)\\=&(k+1)\lim_{r\ra 0^+}W_k(B_r)+(n+1-k)\lim_{r^\ast \ra 0^+} W_{n-k}(\{0\leq r\leq r^\ast\})\\
    =&\left\{ \begin{aligned}
    &(-1)^\frac{n-k-1}{2}\frac{\omega_n}{n+1}\frac{(n-k+1)!!(k+1)!!}{n!!},  \quad &\text{if $n-k$ is odd};\\
    &0,  \quad &\text{otherwise}.
    \end{aligned}  \right.
    \end{align*}
    This shows that the following identities 
    \begin{align*}
    (k+1)W_k(\hat{M})+(n+1-k)W_{n-k}(\hat{M}^\ast)=C_{n,k}, \quad k=0,1,\cdots,n
    \end{align*}
    hold for all pairs $(\hat{M},\hat{M}^\ast)$ of strictly convex bodies in $\mathbb H^{n+1}$ and their dual bodies in $\mathbb S^{n,1}_{+}$, where $C_{n,k}$ are constants depending only on $n$, $k$. If $B_r$ is a geodesic ball in $\mathbb H^{n+1}$ such that $W_k(B_r)=W_k(\hat{M})$, then $W_{n-k}(\hat{M}^\ast)=W_{n-k}(\{0\leq r\leq r^\ast\})$, where $\{0\leq r\leq r^\ast\}$ is the dual body of $B_r$. Note that the geodesic sphere $\partial B_r$ in $\mathbb S^{n+1}$ is strictly convex, while $\{r=r^\ast\}$ in $\mathbb S^{n,1}_{+}$ is its polar hypersurface. Moreover, the principal curvatures of $\partial B_r$ and $\{r=r^\ast\}$ are
    \begin{align*}
    \k_i=\coth r, \quad  \k^{\ast}_i=\tanh r^{\ast}. 
    \end{align*}
    In view of the duality \eqref{s2:duality-hyperbolic} and \eqref{s2:duality-deSitter} in hyperbolic/de Sitter space, we obtain 
    \begin{align*}
    \coth \zeta_k(K) \cdot \tanh \zeta_{n-k}(K^\ast)=1, \quad 0\leq k \leq n.
    \end{align*}
    This completes the proof of Theorem \ref{thm-hyperbolic}.
\end{proof}

\section{Proofs of Theorems \ref{cor-BS-ineq-sphere-extension}--\ref{cor-BS-ineq-deSitter-quermassintegral-ineq-II}}\label{sec:5}
The quermassintegral inequalities \cite{Alexandrov1937,Alexandrov1938} for convex domains in Euclidean space are fundamental in classical geometry. There have been some interests in extending the original quermassintegral inequalities to non-convex domains (see \cite{Trudinger1994}). In \cite{GL09}, Guan and Li extend these inequalities to non-convex starshaped domains in $\mathbb R^{n+1}$. There are generalizations of such inequalities to other ambient spaces, such as the sphere \cite{Makowski-Scheuer2016,CGLS21,Chen-Sun21,GL21,Wei-Xiong2015}, the hyperbolic space \cite{ACW2018,AHL2020,BGL19,HL19,Hu-Li-2022,Hu-Li-Wei2020,Wang-Xia14,Wei-Xiong2015,GeWW14,GL21} and the de Sitter space \cite{Scheuer19,Lambert-Scheuer2021}. We collect these results as follows.

\begin{lem}\label{s5:key-lemma-ineq-I}
	If $K$ is a smooth bounded and strictly convex domain in $\mathbb S^{n+1}$, then the following inequalities hold:
		\begin{enumerate}[(i)]
		\item \cite{BGL19,GL21}
		\begin{align}\label{quermassintegral-ineq-sphere-1}
	    W_n(K)\geq f_n \circ f_l^{-1}(W_l(K)),  \quad 0\leq l<n,
		\end{align}	
		\item \cite{CGLS21,Chen-Sun21}
		\begin{align}\label{quermassintegral-ineq-sphere-2}
		W_k(K) \geq f_k\circ f_0^{-1}(W_0(K)),  \quad 0\leq k<n,
		\end{align}
		\item \cite{Chen-Sun21}
		\begin{align}\label{quermassintegral-ineq-sphere-II}
		W_{m+1}(K) \geq f_{m+1}\circ f_{m-1}^{-1}(W_{m-1}(K)), \quad 1\leq m \leq n-1.
		\end{align}	
		\end{enumerate}
		Equality holds if and only if $K$ is a geodesic ball.
\end{lem}
\begin{rem}
	In view of the identities \eqref{s1:identity-sphere} in Theorem \ref{thm-sphere}, the inequality \eqref{quermassintegral-ineq-sphere-1} and the inequality \eqref{quermassintegral-ineq-sphere-2} are equivalent.
\end{rem}

\begin{lem}\label{s5:key-lemma-ineq-II}
\begin{enumerate}[(a)]
\item Let $K$ be a smooth bounded domain in $\mathbb H^{n+1}$. 
\begin{enumerate}[(i)]
\item \cite{Wang-Xia14,Hu-Li-Wei2020} If $\partial K$ is h-convex, then 
\begin{align}\label{quermassintegral-ineq-hyperbolic}
W_k(K) \geq f_k \circ f_l^{-1}(W_l(K)), \quad 0\leq l<k\leq n.
\end{align}
\item \cite{AHL2020,BGL19} If $\partial K$ is strictly convex, then 
\begin{align}\label{quermassintegral-ineq-hyperbolic-strictly-convex}
W_n(K) \geq f_n \circ f_l^{-1}(W_l(K)), \quad 0\leq l<n,
\end{align}
and
\begin{align}\label{quermassintegral-ineq-hyperbolic-strictly-convex-III}
W_{n-1}(K) \geq f_{n-1}\circ f_{n-1-2i}^{-1}(W_{n-1-2i}(K)), \quad 0<2i<n.
\end{align}
\item \cite{BGL19,GL21} If $\partial K$ is mean convex and starshaped, then 
\begin{align}\label{quermassintegral-ineq-hyperbolic-1-convex}
W_2(K) \geq f_2 \circ f_1^{-1}(W_1(K)).
\end{align}
\item \cite{LWX14} If $\partial K$ is $2$-convex and starshaped, then 
\begin{align}\label{quermassintegral-ineq-hyperbolic-2-convex}
W_3(K) \geq f_3 \circ f_1^{-1}(W_1(K)).
\end{align}
\end{enumerate}
Equality holds if and only if $K$ is a geodesic ball.

\item Let $K$ is a smooth bounded domain in $\mathbb S^{n,1}_{+}$.
\begin{enumerate}[(i)]
	\item \cite{Lambert-Scheuer2021} If $\partial K$ is spacelike, then 
	\begin{align}\label{quermassintegral-ineq-deSitter-I}
	W_0(K) \geq f_0\circ f_1^{-1}(W_1(K)).
	\end{align}
	\item \cite{Scheuer19} If $\partial K$ is mean convex and spacelike, then
\begin{align}\label{quermassintegral-ineq-deSitter-II}
W_1(K) \geq f_1\circ f_2^{-1}(W_2(K)).
\end{align}
\end{enumerate}
Equality holds if and only if $\partial K$ is isometric to a coordinate slice.
\end{enumerate}
\end{lem}

\begin{proof}[Proof of Theorem \ref{cor-BS-ineq-sphere-extension}]
By Theorem \ref{thm-sphere}, we have
\begin{align}\label{s4:identity-in-sphere}
\tan\(\zeta_{n-k}(K)\)\cdot \tan\(\zeta_{k}(K^\ast)\)=1, \quad 0\leq k\leq n,
\end{align}
for smooth bounded and strictly convex domain $K\subset \mathcal{H}(z)$ and its dual body $K^\ast\subset \mathcal{H}(-z)$. Notice that $K^\ast$ is also a smooth bounded and strictly convex domain in the sphere. Then the inequalities \eqref{s1:cor-sphere-ineq-1} and \eqref{s1:cor-sphere-ineq-2} follow from the identity \eqref{s4:identity-in-sphere} and the quermassintegral inequalities \eqref{quermassintegral-ineq-sphere-1} and \eqref{quermassintegral-ineq-sphere-2} in the sphere:
\begin{align}\label{s4:identity-quermassintegral-ineq-sphere}
\zeta_k(K^\ast) \geq \zeta_0(K^\ast), \quad \zeta_{n-k}(K) \leq \zeta_n(K), \quad 0<k\leq n.
\end{align} 
Similarly, the inequalities \eqref{s1:cor-sphere-ineq-3} and \eqref{s1:cor-sphere-ineq-4} follows from the identities \eqref{s4:identity-in-sphere}
and the quermassintegral inequalities \eqref{quermassintegral-ineq-sphere-II}. If equality holds in any one of \eqref{s1:cor-sphere-ineq-1}--\eqref{s1:cor-sphere-ineq-4}, then the equality also holds in the quermassintegral inequalities \eqref{quermassintegral-ineq-sphere-1}, \eqref{quermassintegral-ineq-sphere-2} or \eqref{quermassintegral-ineq-sphere-II}, which implies that $K$ (or $K^\ast$) is a geodesic ball in $\mathbb S^{n+1}$. This completes the proof of Theorem \ref{cor-BS-ineq-sphere-extension}.
\end{proof}

\begin{proof}[Proof of Theorem \ref{cor-BS-ineq-hyperbolic-extension}]
	By Theorem \ref{thm-hyperbolic}, we have the following identities
	\begin{align}\label{s4:identity-in-hyperbolic}
	\coth(\zeta_k(K)) \cdot \tanh(\zeta_{n-k}(K^\ast))=1, \quad 0\leq k\leq n,
	\end{align}
	for any smooth bounded and strictly convex domain $K$ in $\mathbb H^{n+1}$ and its dual body $K^\ast$ in $\mathbb S^{n,1}_{+}$. For any $0\leq k,l\leq n$ with $k+l<n$, then $0\leq k<n-l$ and the quermassintegral inequalities \eqref{quermassintegral-ineq-hyperbolic} for h-convex domains in $\mathbb H^{n+1}$ imply that
	\begin{align}\label{s4:identity-quermassintegral-ineq-hyperbolic}
	\zeta_k(K) \leq \zeta_{n-l}(K).
	\end{align}
	Note that $\coth s$ is monotone decreasing in $s$, we get
	\begin{align*}
	\coth(\zeta_k(K)) \cdot \tanh(\zeta_{l}(K^\ast)) \geq \coth(\zeta_{n-l}(K)) \cdot \tanh(\zeta_{l}(K^\ast))=1,
	\end{align*}
	which yields the desired inequality \eqref{s1:cor-hyperbolic-ineq-1}. The inequality \eqref{s1:cor-hyperbolic-ineq-2} can be proved similarly. If the equality holds in \eqref{s1:cor-hyperbolic-ineq-1} or \eqref{s1:cor-hyperbolic-ineq-2}, then the equality also holds in \eqref{s4:identity-quermassintegral-ineq-hyperbolic}, which implies that $K$ is a geodesic ball in $\mathbb H^{n+1}$. This completes the proof of Theorem \ref{cor-BS-ineq-hyperbolic-extension}.	
\end{proof}

\begin{proof}[Proof of Theorem \ref{cor-BS-ineq-hyperbolic-quermassintegral-ineq}]
    By \eqref{quermassintegral-ineq-hyperbolic-strictly-convex-III}, for any smooth bounded and strictly convex domain $K$ in $\mathbb H^{n+1}$, there hold
    \begin{align*}
    W_{n-1-2i}(K) \leq f_{n-1-2i}\circ f^{-1}_{n-1}(W_{n-1}(K)), \quad 0<2i<n.
    \end{align*}
    To complete the proof, it remains to show that 
    \begin{align}\label{s5:claim}
    W_{n-2k}(K) \leq f_{n-2k}\circ f^{-1}_{n-1}(W_{n-1}(K)), \quad 0<2k\leq n.
    \end{align}
    We prove \eqref{s5:claim} by induction. For $k=1$, by taking an interior point of $K$ as the Beltrami point, its dual $K^\ast$ is a strictly convex and spacelike domain in $\mathbb S^{n,1}_{+}$. Then the quermassintegral inequality \eqref{quermassintegral-ineq-deSitter-II} in $\mathbb S_{+}^{n,1}$ implies that
    $$
    \tanh(\zeta_{2}(K^\ast)) \leq \tanh(\zeta_{1}(K^\ast)).
    $$
    Combining this with the identities \eqref{s4:identity-in-hyperbolic}, i.e.,
    $$
    \coth(\zeta_{n-2}(K)) \cdot \tanh(\zeta_{2}(K^\ast))=\coth(\zeta_{n-1}(K)) \cdot \tanh(\zeta_{1}(K^\ast)),
    $$
    we obtain
    \begin{align*}
    \zeta_{n-2}(K) \leq \zeta_{n-1}(K),
    \end{align*}
    in view of the monotonicity of $\coth s$. This verifies \eqref{s5:claim} for $k=1$.
    
    Assume that \eqref{s5:claim} is true for $k-1$, i.e., 
    \begin{align}\label{s5:induction}
    W_{n-2k+2}(K) \leq f_{n-2k+2}\circ f^{-1}_{n-1}(W_{n-1}(K)),
    \end{align}
    we show that \eqref{s5:claim} also holds for $k$. Along the harmonic mean curvature flow starting from $\partial K$ in $\mathbb H^{n+1}$, we know that the flow hypersurface $M_t=\partial K_t$ is strictly convex and it shrinks to a round point as $t\ra T^\ast$. Then by \eqref{s3:variational-formula}, we have
    \begin{align}\label{s5:evol-Wn-1-radial}
    \frac{d}{dt}W_{n-1}(K_t)=&-\frac{2}{n+1}\int_{M_t}E_n d\mu \nonumber\\ 
    =&-2W_{n+1}(K_t)-nW_{n-1}(K_t) \nonumber\\ 
    =&\left.\(-2W_{n+1}(B_r)-nW_{n-1}(B_r)\)\right|_{r=f_{n-1}^{-1}(W_{n-1}(K_t))}\nonumber\\ 
    =&-\left.\frac{2}{n+1}\int_{\partial B_r}E_n d\mu \right|_{r=f_{n-1}^{-1}(W_{n-1}(K_t))}\nonumber\\ 
    =&-\left.\frac{2}{n+1}\omega_n \cosh^n r \right|_{r=f_{n-1}^{-1}(W_{n-1}(K_t))},
    \end{align}
    where we used \eqref{s1:quermassintegral-def} and 
    $$
    W_{n+1}(K_t)=W_{n+1}(B_r)|_{r=f_{n-1}^{-1}(W_{n-1}(K_t))}=\frac{\omega_n}{n+1}
    $$
    due to Gauss-Bonnet-Chern theorem (see \cite{Chern1944,Chern1945}, see also \cite{Solanes06}). 
    For $k\ge 2$, we also have
    \begin{align}\label{s5:evol-Wn-2k}
    \frac{d}{dt}W_{n-2k}(K_t)=&-\frac{2k+1}{n+1}\int_{M_t} E_{n-2k}\frac{E_n}{E_{n-1}}d\mu \nonumber\\
    \geq &-\frac{2k+1}{n+1}\int_{M_t} E_{n-2k+1} d\mu \nonumber\\
    =&-(2k+1)W_{n-2k+2}(K_t)-(n+1-2k)W_{n-2k}(K_t) \nonumber\\
    \geq & -(2k+1)f_{n-2k+2}\circ f^{-1}_{n-1}(W_{n-1}(K_t))-(n+1-2k)W_{n-2k}(K_t),
    \end{align}
    where we used the Newton-MacLaurin inequality, \eqref{s1:quermassintegral-def} and the induction assumption \eqref{s5:induction}. Note that $f_m(r)=W_m(B_r)$, we have
    \begin{align}\label{s5:W_m-diff}
    f'_m(r)=\frac{d}{dr}W_m(B_r)=&\frac{n+1-m}{n+1}\int_{\partial B_r} E_m d\mu \nonumber\\
           =&\frac{n+1-m}{n+1}\omega_n \sinh^{n-m}r \cosh^{m}r, \quad 0\leq m\leq n. 
    \end{align}
    Using \eqref{s5:W_m-diff} and \eqref{s5:evol-Wn-1-radial}, we get
    \begin{align*}
    &\frac{d}{dt}f_{n-2k}\circ f_{n-1}^{-1}(W_{n-1}(K_t))\\=&\left.\frac{f'_{n-2k}(r)}{f'_{n-1}(r)}\right|_{r=f_{n-1}^{-1}(W_{n-1}(K_t))}\cdot \frac{d}{dt}W_{n-1}(K_t)\\
    =&\frac{2k+1}{2}\left.\frac{\sinh^{2k-1}r}{\cosh^{2k-1}r} \cdot  \(-\frac{2}{n+1}\omega_n \cosh^{n} r\)\right|_{r=f_{n-1}^{-1}(W_{n-1}(K_t))} \\
    =&-\left.\frac{2k+1}{n+1}\int_{\partial B_r}E_{n+1-2k}d\mu \right|_{r=f_{n-1}^{-1}(W_{n-1}(K_t))}\\
    =&-(2k+1) f_{n+2-2k}\circ f_{n-1}^{-1}(W_{n-1}(K_t))-(n+1-2k) f_{n-2k}\circ  f_{n-1}^{-1}(W_{n-1}(K_t)),
    \end{align*}
    where we also used \eqref{s1:quermassintegral-def} in the last equality. Combining this with \eqref{s5:evol-Wn-2k}, we obtain
    \begin{align*}
          &\frac{d}{dt}\(f_{n-2k}\circ f_{n-1}^{-1}(W_{n-1}(K_t))-W_{n-2k}(K_t)\) \\
    \leq  &-(n+1-2k)\(f_{n-2k}\circ f_{n-1}^{-1}(W_{n-1}(K_t))-W_{n-2k}(K_t)\),
    \end{align*}
    and hence 
    \begin{align*}
    \frac{d}{dt} \left[e^{(n+1-2k)t}\(f_{n-2k}\circ f_{n-1}^{-1}(W_{n-1}(K_t))-W_{n-2k}(K_t)\)\right]\leq 0.
    \end{align*}
    As the flow hypersurface $M_t$ shrinks to a round point as $t\ra T^{\ast}$, and the quermassintegral is monotone under the set inclusion, we have (see \cite[Lemma 3.3]{AHL2020})
    \begin{align*}
    \lim_{t\ra T^\ast}W_{m}(K_t)=0,\quad 0\leq m\leq n.
    \end{align*}
    It follows that
    \begin{align*}
         &f_{n-2k}\circ f_{n-1}^{-1}(W_{n-1}(K))-W_{n-2k}(K) \\
    \geq &e^{(n+1-2k)T^\ast}\lim_{t\ra T^\ast}\left[f_{n-2k}\circ f_{n-1}^{-1}(W_{n-1}(K_t))-W_{n-2k}(K_t)\right]\\
    =&0,
    \end{align*} 
    which verifies \eqref{s5:claim} for $k\geq 2$. The equality case can be similarly characterized. This completes the proof of Theorem \ref{cor-BS-ineq-hyperbolic-quermassintegral-ineq}.   
\end{proof}

\begin{proof}[Proof of Theorem \ref{cor-BS-ineq-hyperbolic-extension-II}]
	The quermassintegral inequalities \eqref{quermassintegral-ineq-hyperbolic-strictly-convex-II} in Theorem \ref{cor-BS-ineq-hyperbolic-quermassintegral-ineq} and \eqref{quermassintegral-ineq-hyperbolic-strictly-convex} in Lemma \ref{s5:key-lemma-ineq-II} for strictly convex domains in $\mathbb H^{n+1}$ imply that
	\begin{align*}
	\zeta_l(K) \leq & \zeta_{n-1}(K), \quad 0\leq l<n-1,
	\end{align*}
	and
	\begin{align*}
	\zeta_k(K) \leq & \zeta_{n}(K), \quad 0\leq k<n.
	\end{align*}
	In view of \eqref{s4:identity-in-hyperbolic} and the monotonicity of $\cosh s$, we get
	\begin{align*}
	\coth(\zeta_l(K)) \cdot \tanh(\zeta_{1}(K^\ast)) \geq \coth(\zeta_{n-1}(K)) \cdot \tanh(\zeta_{1}(K^\ast))=1, \quad 0\leq l<n-1,
	\end{align*}
	and
	\begin{align*}
	\coth(\zeta_k(K)) \cdot \tanh(\zeta_{0}(K^\ast)) \geq \coth(\zeta_{n}(K)) \cdot \tanh(\zeta_{0}(K^\ast))=1, \quad 0\leq k<n,
	\end{align*}
	which are the desired inequalities \eqref{s1:hyperbolic-ineq-1} and \eqref{s1:hyperbolic-ineq-2}, respectively. If the equality holds in \eqref{s1:hyperbolic-ineq-1} or \eqref{s1:hyperbolic-ineq-2}, then the equality also holds in \eqref{quermassintegral-ineq-hyperbolic-strictly-convex-II} or \eqref{quermassintegral-ineq-hyperbolic-strictly-convex}, which implies that $K$ is a geodesic ball in $\mathbb H^{n+1}$. This completes the proof of Theorem \ref{cor-BS-ineq-hyperbolic-extension-II}.	
\end{proof}

\begin{proof}[Proof of Theorem \ref{cor-BS-ineq-deSitter-quermassintegral-ineq}]
	For any smooth bounded spacelike domain $K$ in $\mathbb S^{n,1}_{+}$ with its principal curvatures satisfy $0<\k_i\leq 1$, then its dual body $K^\ast$ is a smooth bounded and h-convex domain in hyperbolic space. In view of duality in hyperbolic/de Sitter space, we have $(K^\ast)\ast=K$. The quermassintegral inequality \eqref{quermassintegral-ineq-hyperbolic} for smooth bounded and h-convex domains in hyperbolic space implies that
	\begin{align*}
	\zeta_{l}(K^\ast) \leq \zeta_{k}(K^\ast), \quad 0\leq l<k\leq n.
	\end{align*}
	Then by the monotonicity of $\coth s$, we get
	$$
	\coth(\zeta_{l}(K^\ast)) \geq \coth(\zeta_{k}(K^\ast)), \quad 0\leq l<k\leq n.
	$$
	In view of the identities \eqref{s1:identity-hyperbolic}, i.e.,
	$$
	\coth(\zeta_{l}(K^\ast)) \cdot \tanh(\zeta_{n-l}(K))= \coth(\zeta_{k}(K^\ast)) \cdot \tanh(\zeta_{n-k}(K)), \quad 0\leq l<k\leq n,
	$$
	we obtain
	\begin{align*}
	\zeta_{n-l}(K) \leq \zeta_{n-k}(K), \quad 0\leq l<k\leq n,
	\end{align*}
	which is equivalent to \eqref{s2:quermassintegral-ineq-I}. The equality case can be similarly characterized. This completes the proof of Theorem \ref{cor-BS-ineq-deSitter-quermassintegral-ineq}.
\end{proof}

\begin{proof}[Proof of Theorem \ref{cor-BS-ineq-deSitter-quermassintegral-ineq-II}]
    For any smooth bounded spacelike and strictly convex domain $K$ in $\mathbb S^{n,1}_{+}$, its dual body $K^\ast$ is a smooth bounded and strictly convex domain in hyperbolic space. Then by the monotonicity of $\coth s$ and the quermassintegral inequalities \eqref{quermassintegral-ineq-hyperbolic-strictly-convex-II}, we obtain 
	\begin{align*}
	\coth(\zeta_{n-1}(K^\ast)) \leq \coth(\zeta_{k}(K^\ast)), \quad 0\leq k<n-1.
	\end{align*}
	In view of the identities \eqref{s1:identity-hyperbolic}, we have
	\begin{align*}
	\coth(\zeta_{n-1}(K^\ast)) \cdot \tanh(\zeta_{1}(K))=\coth(\zeta_{k}(K^\ast)) \cdot \tanh(\zeta_{n-k}(K)), \quad 0\leq k<n-1. 
	\end{align*}
	It follows that
	\begin{align*}
	\tanh(\zeta_{1}(K)) \geq \tanh(\zeta_{n-k}(K)), \quad 0\leq k<n-1,
	\end{align*} 
	which is equivalent to \eqref{s2:quermassintegral-de-Sitter-ineq-I} due to the monotonicity of $\tanh s$. 
	Applying the quermassintegral inequalities \eqref{quermassintegral-ineq-hyperbolic-1-convex} and \eqref{quermassintegral-ineq-hyperbolic-2-convex} to strictly convex hypersurfaces in hyperbolic space, the inequalities \eqref{s2:quermassintegral-de-Sitter-ineq-II} and \eqref{s2:quermassintegral-de-Sitter-ineq-III} can be proved similarly. Moreover, the equality case in \eqref{s2:quermassintegral-de-Sitter-ineq-I}--\eqref{s2:quermassintegral-de-Sitter-ineq-III} can be similarly characterized. This completes the proof of Theorem \ref{cor-BS-ineq-deSitter-quermassintegral-ineq-II}. 	
\end{proof}

\end{document}